\documentclass{article}

\usepackage{arxiv}
\usepackage[utf8]{inputenc} 
\usepackage[T1]{fontenc}    

\RequirePackage{amsthm,amsmath,amsfonts,amssymb}
\RequirePackage[colorlinks, linkcolor=black, citecolor=blue, urlcolor=blue]{hyperref}
\usepackage{cleveref}
\RequirePackage{graphicx} 
\RequirePackage{subcaption}
\RequirePackage{enumitem}
\RequirePackage[ruled]{algorithm2e}

\numberwithin{equation}{section}
\theoremstyle{plain}
\newtheorem{theorem}{Theorem}
\newtheorem{proposition}[theorem]{Proposition}
\newtheorem{lemma}{Lemma}
\theoremstyle{remark}
\newtheorem{definition}{Definition}

\crefformat{footnote}{#2\footnotemark[#1]#3}

\newcommand{\abs}[1]{\left|#1\right|} 
\newcommand{\norm}[1]{\left\Vert#1\right\Vert} 

\newcommand{\proba}[2]{\mathbb{P}_{#1}\left(#2 \right)}
\newcommand{\expect}[2]{\mathbb{E}_{#1}\left[ #2 \right]}

\renewcommand{\epsilon}{\varepsilon}

\newcommand{\C}{\mathcal{C}}

\newcommand{\F}{\mathcal{F}}

\newcommand{\M}{\mathcal{M}}

\newcommand{\X}{\mathcal{X}}

\newcommand{\bbG}{\mathbb{G}}
\newcommand{\bbX}{\mathbb{X}}

\newcommand{\Q}{\mathbb{Q}}
\newcommand{\R}{\mathbb{R}}

\newcommand{\as}{\textit{a.s.}}

\newcommand{\lipsch}{\textbf{(L)} }
\newcommand{\bound}{\textbf{(B)} }

\newcommand{\nnodes}{n}
\newcommand{\Gn}{\mathcal{G}_{\nnodes}}
\newcommand{\Gun}{\mathcal{G}^{\nnodes}}

\newcommand{\ER}{ER}

\newcommand{\db}{d_B}
\newcommand{\dg}{Dg}

\newcommand{\segm}{[0,T]}
\newcommand{\FHKD}{\mathcal{F}_{HKD}}
\newcommand{\FHPD}{\mathcal{F}_{HPD}}

\title{Heat diffusion distance processes : a statistically founded method to analyze graph data sets}
\renewcommand{\shorttitle}{Heat diffusion distance processes}

\author{
 Etienne Lasalle, \\
  Laboratoire de Mathématiques d'Orsay,\\
  CNRS, Université Paris-Saclay, \\
  Orsay, France\\
  \url{etienne.lasalle@universite-paris-saclay.fr} }

\begin{document}

\maketitle

\begin{abstract}
We propose two multiscale comparisons of graphs using heat diffusion, allowing to compare graphs without node correspondence or even with different sizes. These multiscale comparisons lead to the definition of Lipschitz-continuous empirical processes indexed by a real parameter. The statistical properties of empirical means of such processes are studied in the general case. Under mild assumptions, we prove a functional central limit theorem, as well as a Gaussian approximation with a rate depending only on the sample size. Once applied to our processes, these results allow to analyze data sets of pairs of graphs. We design consistent confidence bands around empirical means and consistent two-sample tests, using bootstrap methods. Their performances are evaluated by simulations on synthetic data sets.
\end{abstract}

\section{Introduction}

Considering the current growth of available data and the modeling power of networks, methods to analyze graph-structured data have gained interest over the last few decades. Particular attention has been devoted to designing notions of distance between graphs.  
The design of these notions is highly constrained by the working framework. In particular, different types of information can be used depending on whether the graphs are directed or undirected, weighted or unweighted, have the same size or not. Another key factor to define distances is whether a node correspondence (NC) is known or not.
In the case of known NC, we can consider the graphs to be defined on the same vertex set and comparisons can be made at the edge scale. In this context, people have applied various metrics to compare adjacency matrices, Laplacian matrices, heat kernels \cite{hammond2013graph} and other matrices whose entries represent quantities associated to pairs of nodes \cite{koutra2013deltacon}.
On the other hand, when no NC is available, graphs are often compared at a mesoscopic or macroscopic level using structural summaries. People have used global statistics on graphs like degree distributions, network diameters, or clustering coefficients \cite{prvzulj2004modeling}.
Another well studied approach is to consider graphlets \cite{prvzulj2004modeling}, \textit{i.e.}, small given subgraphs that are counted in graphs. Then, various methods have been developed to compare the graphlet counts \cite{prvzulj2007biological, yaverouglu2014revealing, ali2014alignment, faisal2017grafene}.
Some work has also been pursued to exploit the structural information carried by spectra of operators \cite{wilson2008study, gera2018identifying}. 
To measure distances (or equivalently similarities) between graphs, one can also rely on graph kernels, hence benefiting from the general kernel methods to solve statistical problems. These various kernels can be based on neighborhood information, subgraphs structures, random walk properties, shortest paths. Unfortunately, these kernels often require additional information like node or edge labels, attributes, which are not always available. Moreover, not all graph kernels are able to handle weighted graphs. For more details on graph kernels, see the survey of \cite{kriege2020survey}.

Another powerful way to encode structural information about graphs is to use diffusion processes, like heat diffusion.
When working with weighted graphs, one can interpret weights as the thermal conductivity of edges, meaning that heat diffuses faster along edges with higher weights.
Note that unweighted graphs can always be seen as weighted graphs with weights in $\{0,1\}$.
Given initial conditions, the way heat diffuses can be used to characterize and compare graphs \cite{coifman2006diffusion, coifman2014diffusion, hammond2013graph, tsitsulin2018netlsd}.
This approach is appealing as it allows to analyze graphs at different scales by looking at different diffusion times $t$.
For small values of $t$, the diffusion only concerns a small neighborhood of the initially heated nodes, while for larger values it involves larger and possibly more complex structures, taking into account topological properties of the graph. Thus, the choice of relevant and informative diffusion times is essential.

For more references on comparisons of graphs, we refer the reader to \cite{soundarajan2014guide}, \cite{emmert2016fifty}, \cite{tantardini2019comparing} and references therein.

\subsection{Our contributions}
While a lot of the above notions of distances are often supported by experimental results and applications to learning or data mining tasks, they usually lack statistical foundations.
In this context, we provide new tools to analyze and compare graphs or even data sets of graphs, that benefit from statistical guarantees.
Our methods take advantage of the desirable multiscale property of heat diffusion. Moreover, one of our methods can deal with graphs without known NC or even graphs of different sizes, by using topological descriptors from topological data analysis (TDA).
To circumvent the difficulty of choosing a suitable diffusion time, we opt to take into account the whole diffusion process.
As a result, we define two real-valued processes, indexed by all the diffusion times in $\segm$ for some $T>0$, representing comparisons of heat distributions.
Basically, instead of trying to define a real value that would measure the distance between two graphs, we define a curve, a distance profile, that represents the comparisons of heat distributions across all diffusion times. 

The first process, called Heat Kernel Distance (HKD) process, is defined by comparing heat kernels with the Frobenius norm. In this case, a NC between graphs needs to be known for the entry-wise comparison of the heat kernels to be meaningful.
The second process, called Heat Persistence Distance (HPD), is defined using tools from TDA and can deal with graphs of different sizes.
To do so, each graph is equipped with a real-valued function defined on the vertex set: the Heat Kernel Signature (HKS) \cite{sun2009concise, hu2014stable}.
Then, graphs are converted into topological descriptors called persistence diagrams. They are multisets of points in $\R^2$, encoding how topological features, like connected components and loops, evolve along with the families of sublevel and superlevel subgraphs. The diagrams are then compared with the so-called Bottleneck distance.
Using persistence diagrams allows to switch from node-based representations of graphs to comparable topological summaries, hence requiring no assumption on graph sizes and NC.

To statistically study the HKD and HPD processes, we prove general results on Lipschitz-continuous real-valued empirical processes indexed by a real parameter. Namely, we show that they verify a functional Central Limit Theorem and admit Gaussian approximations with rates depending only on the sample sizes. These results ensure the asymptotic validity of bootstrap methods to design confidence bands around empirical mean processes, as well as consistent two-sample tests. 
They are applied to the HKD and HPD processes and could be applied to any other Lipschitz-continuous processes indexed by a real parameter under mild assumptions. 
In our graph framework, we can summarize the general idea as follows. We convert families of pairs of graphs into families of distance curves using HKD or HPD processes and then, we leverage the statistical properties of smooth curves mentioned above to exploit these families of curves.  
As a proof of concept, we illustrate these results on simulated data sets of pairs of graphs, drawn from various models: Erd\H{o}s-Rényi model, stochastic block model, and random geometric graph models. We also compare our HKD-based test with tests based on other notions of distance and show improvements when using the multi-scale comparison of graphs. 
The HPD-based test is compared with various tests based on graph kernels.

The results from this article allow to statistically analyze data sets of pairs of graphs. While it might seem more natural to study data sets of graphs, we would like to motivate our analysis. There are various situations where graphs naturally come in pairs, and where the relevant information is actually the structural changes inside the pairs. 
For example, when monitoring brain diseases \cite{rocca2016impaired, farahani2019application}, patients'~brains can be observed at different time points. In this case, the changes of the brains' connectivity are of prime interest to assess the evolution of the disease \cite{faivre2016depletion, castellazzi2018functional}.
Moreover, even when graphs don't come in pairs, there exist several ways to turn a data set of graphs into a data set of pairs of graphs. For example, one can consider all the possible pairs of graphs from the data set or split the data set into two groups and construct pairs that contain a graph of each group. 
However, designing relevant approaches to combine graphs into pairs is probably a task-dependent problem, which is beyond the scope of this article.

\subsection{Organisation}
The rest of the paper is organized as follows. Section~\ref{sec:heat_dist_proc} introduces the graph framework, heat diffusion on graphs, the HKD and HPD processes. General results of such processes are developed in Section~\ref{sec:emp_proc}. We introduce the framework for general continuous real-valued empirical processes indexed by a real parameter, prove their statistical properties and present some applications on bootstrap methods. These general results are applied to the HKD and HPD processes in Section~\ref{sec:result_hd}, where we also detail the construction of a HKD or HPD-based confidence bands and two-sample tests for samples of pairs of graphs. We provide asymptotic results on the level and power of these tests. Finally, as a proof of concept, we illustrate the construction of confidence bands and two-sample tests in Section~\ref{sec:experiments} using several generative models of random graphs. All Python codes are available at \url{https://github.com/elasalle/HeatDistanceProcess}.

\section{Comparing graphs using heat distance processes}
\label{sec:heat_dist_proc}

In this section, we introduce notions of distance between graphs based on heat diffusion. We define the resulting HKD and HPD processes.
\subsection{Background and definitions}

Before introducing the key notions of this work, we present general definitions and notations that will be used in the rest of the paper. We start by introducing notations relative to graph theory before presenting the theory of extended persistence adapted to graphs.

\subsubsection{Graphs}
For $0 \leq w_{\min} \leq w_{\max}$, we denote by $\Gn(w_{\min}, w_{\max})$ the set of undirected weighted graphs of size $\nnodes$, without self-loop and whose weights are in $\{0\} \cup [w_{\min}, w_{\max}]$. The special case of unweighted graphs correspond to $w_{\min}=w_{\max}=1$. For clarity in the notation, we remove the $w_{\min}$ and  $w_{\max}$, whenever there is no ambiguity.  We also consider $\Gun$ (with $\nnodes$ as an exponent) the set of graphs of size at most $\nnodes$, \textit{i.e.}, $\Gun = \cup_{1 \leq i \leq \nnodes} \mathcal{G}_i$. For a graph $G$ in $\Gn$, we denote by $W(G)$ its weight matrix (or adjacency matrix), \textit{i.e.}, the $\nnodes \times \nnodes$ symmetric matrix whose $(i,j)$-coefficient is the weight $w_{i,j}$ of edge $\{i,j\}$. The degree matrix $D(G)$ denotes the diagonal matrix whose entry $D(G)_{i,i}$ is the degree of node $i$ defined by $\sum_j w_{i,j}(G)$, the sum of all incident weights. The combinatorial Laplacian $L(G)$ is defined by $D(G)-W(G)$. Taking non-negative weights ensures that $L(G)$ is a real symmetric positive-semidefinite matrix. From now on, we forget the dependence in $G$ in the notation, whenever there is no ambiguity. Let $\lambda_1 \leq \dots \leq \lambda_\nnodes$ be the eigenvalues of $L$ and let $(\phi_1, \dots, \phi_{\nnodes})$ be a family of orthonormal eigenvectors. We denote by $\Lambda$ the diagonal matrix containing the eigenvalues on the diagonal and $\phi$ the matrix whose columns are the $\phi_i$'s so that $L$ admits the following decomposition
\begin{equation}
L = \phi \Lambda \phi^T = \sum\limits_{k = 1}^\nnodes \lambda_k \phi_k \phi_k^T.
\label{eq:laplacian_spectral_decomposition}
\end{equation}

Note that $\lambda_1 = 0$ and that $\phi_1$ can always be chosen to be the vector whose entries are equal to $1 / \sqrt{\nnodes}$. In the following, this choice will always be made.

\subsubsection{Persistence on graphs}
We present here the basics of ordinary and extended persistence. We refer the reader to \cite{cohen2009extending}, \cite{edelsbrunner2010computational} and \cite{oudot2015persistence} for a complete description of these theories.

Persistence theory allows to study the topology of topological spaces in a multiscale manner. Usually, given a topological space $X$ and a continuous real-valued function $f : X \to \R$, one considers the family of sublevel sets $X_\alpha := \{ x \in X , \ f(x) \leq \alpha \}$, for $\alpha$ varying from $-\infty$ to $+\infty$.
Ordinary persistence records the levels at which topological features (connected components, loops, cavities or higher dimensional holes...) appear and disappear.
For each feature, its birth and death levels are stored as the coordinates $(b,d)$ of a point in $\R^2$. The multiset of these points is called a persistence diagram.
This framework can be applied to graphs.

Let $G = (V,E)$ be a graph with vertex set $V$ and edge set $E$, and $f$ be a real-valued function on $V$.
Consider the family of sublevel subgraphs $(G_\alpha)_{\alpha \in \R}$, where $G_\alpha = (V_\alpha, E_\alpha)$ with $V_\alpha = \{ v \in V , \ f(v) \leq \alpha \}$ and $E_\alpha = \{ \{v,v'\} \in E, \ v,v' \in V_\alpha \}$.
Across the family of sublevel graphs, as $\alpha$ increases, we can record birth and death levels of connected components, and birth levels of loops.
As a connected component dies when it gets connected to an older connected component, remark that the connected components of $G$ will never die.
Similarly, as a loop dies when it gets "filled-in" by a 2-dimensional object, loops appearing in $G_\alpha$ (an object of maximal dimension 1) will never die.
To prevent topological features from having no death levels (or infinite death levels), the theory of extended persistence suggests that the family of superlevel subgraphs should also be considered.
Define $G^\alpha = (V^\alpha, E^\alpha)$ similarly to $G_\alpha$ with $V^\alpha = \{ v \in V , \ f(v) \geq \alpha \}$.
A death level is now assigned to connected components of $G$ and loops as the level at which they appear in the family of superlevel subgraphs when $\alpha$ decreases.
Additionally, we record the birth and death of connected components in the family of superlevel subgraphs.
Hence, extended persistence is able to detect four types of topological features and extract their birth and death levels $(b,d)$ corresponding to four types of points:
\begin{itemize}
\item $Ord_0$ : birth and death of a connected component in $(G_\alpha)$.
\item $Rel_1$ : birth and death of a connected component in $(G^\alpha)$.
\item $Ext_0$ : birth and death of a connected component of $G$ when using $(G_\alpha)$ and $(G^\alpha)$.
\item $Ext_1$ : birth and death of a loop when using both $(G_\alpha)$ and $(G^\alpha)$.
\end{itemize}
These four types of topological features can be seen as downward branches, upward branches, connected components, and loops, respectively; with the orientation being taken with respect to $f$.
For a graph $G$ and a function $f$ on its vertices, we will denote by $Ord_0(G,f)$, $Rel_1(G,f)$, $Ext_0(G,f)$ and $Ext_1(G,f)$ the persistence diagrams containing the corresponding points. In the following, $\dg(G,f)$ will generically denote any of these four diagrams.
We refer the reader to \cite[Section 2.1]{carriere2020perslay} for a more precise and illustrative presentation of extended persistence diagrams on graphs.

The space of diagrams can be equipped with the Bottleneck distance $\db$. We recall its definition. Let $\mu$ and $\nu$ be two diagrams, \textit{i.e.}, two multisets of points in $\R^2$, and let $\Delta := \{ (a,a), a \in \R \}$ be the diagonal. Denote by $\Pi(\mu, \nu)$ the set of bijections from $\mu \cup \Delta$ to $\nu \cup \Delta$. $\db$ is defined by
\begin{equation*}
\db(\mu, \nu) := \underset{\pi \in\Pi(\mu, \nu) }{\inf} \ \underset{x \in \mu \cup \Delta}{\sup} \ \| x - \pi(x)\|_\infty.
\end{equation*}
We state a stability result for extended persistence diagrams computed on graphs. It is a consequence of a more general stability result for persistence diagrams \cite{chazal2016structure, cohen2009extending}.
\begin{theorem}
For all graphs $G = (V,E)$, for all $f, f' : V \to \R$, and for all diagram constructions $\dg$ among $Ord_0$, $Rel_1$, $Ext_0$ and $Ext_1$,
\begin{equation*}
\db( \dg(G,f) , \dg(G,f') ) \leq \| f-f' \|_\infty ,
\end{equation*}
with $\|f\|_\infty = \max \{ \abs{f(v)}, \ v \in V \}$.
\label{theo:bottleneck_stability}
\end{theorem}

\subsection{Heat distances}

Let $G$ be a graph in $\Gn$ and let $L$ be its Laplacian. For $t \geq 0$, let $u_t \in \R^\nnodes$ be the vector whose $i$-th coefficient represents the amount of heat of node $i$ at time $t$. Then, $u_t$ follows the heat equation:
\begin{equation*}
\forall t \geq 0, \quad \frac{d}{dt} u_t = - L u_t.
\end{equation*}
The solution is given by $u_t = e^{-tL} u_0$. The matrix $e^{-tL}$ is called the \textit{heat kernel} and describes how heat diffuses in the graph. The $i$-th column of $e^{-tL}$ contains the amount of heat of each node at time $t$, when a single unit of heat was placed at node $i$ at time $t=0$.
From (\ref{eq:laplacian_spectral_decomposition}), the heat kernel decomposes in
\begin{equation*}
e^{-tL} = \phi e^{-t\Lambda} \phi^T = \sum\limits_{k = 1}^\nnodes e^{-t\lambda_k} \phi_k \phi_k^T.
\end{equation*}

The heat kernel encodes all the solutions of the heat equation and will be used to define notions of distances between graphs. Note that the diffusion time $t$ acts as a scale parameter, where diffusion for small values of $t$ only considers direct neighborhoods of the nodes, while for larger values it takes more global structures into account. This multi-scale property will be exploited in the following.

\subsubsection{Heat Kernel Distance}
Assume that we know the NC for graphs in $\Gn$ and that we number the nodes such that the identity mapping gives the correspondences. Hence, comparing adjacency matrices, Laplacians, or heat kernels entry-wise becomes meaningful. Here we compare graphs through their heat kernels.
For two graphs $G$ and $G'$, define their Heat Kernel Distance (HKD) at time $t$ by
\begin{equation}
D_t((G,G')) = \| e^{-tL} - e^{-tL'} \|_F, 
\label{eq:HKD_def}
\end{equation}
where $L$ and $L'$ are the laplacian matrices of $G$ and $G'$ respectively, and $\| \cdot \nobreak \|_F$ denotes the Frobenius norm. 
This notion of distance was introduced by \cite{hammond2013graph}. Basically,  it compares different signals defined on the nodes of $G$ and $G'$. These signals are given by the heat distribution obtained after a diffusion time $t$, for all elementary initial conditions. 
Considering heat kernels instead of  adjacency or Laplacian matrices allows for a more robust notion of distance. 
Indeed, it is more aware of the global structures of graphs. Adding or removing edges with low impacts on the overall graphs structures results in small changes of the HKDs, especially for diffusion times that are not too close to zero. For example removing an edge joining two dense components is seen differently by the HKD than removing an edge inside one of the component 
(see Figure~\ref{fig:example_HKD}).
This would not be the case when comparing adjacency or laplacian matrices.
The example in Figure~\ref{fig:example_HKD} also illustrates the multi-scale property of HKDs. At a local scale the HKD only detects that an edge has been removed. However, at a larger scale the functional role of the removed edge for the graph connectivity is captured by the HKD.
Moreover, we choose to compare the heat kernels through the Frobenius norm as it is a Euclidean norm on the space of matrices. This property will be handy for deriving an exact expression of $D_t((G,G'))$ in terms of the eigen-elements of $L$ and $L'$ in Proposition~\ref{prop:hkd}. 

\begin{figure}
\centering
\includegraphics[width=.8\linewidth]{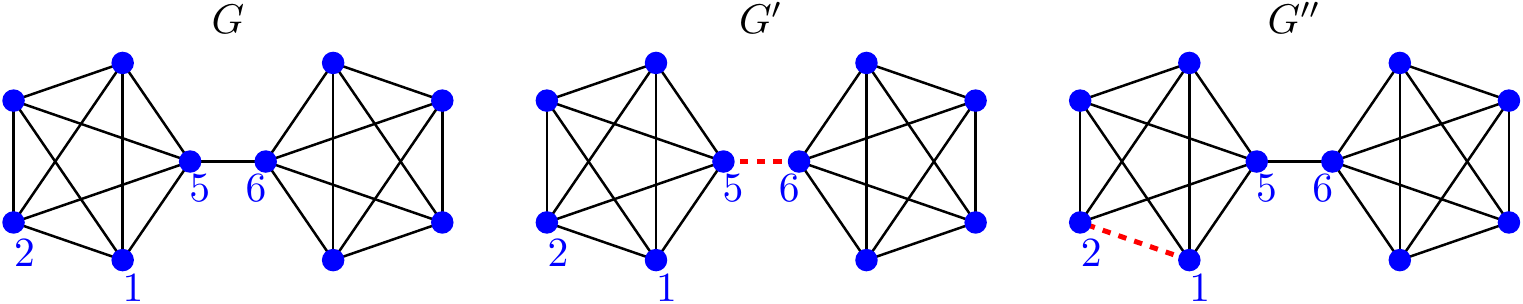}
\includegraphics[width=0.4\linewidth]{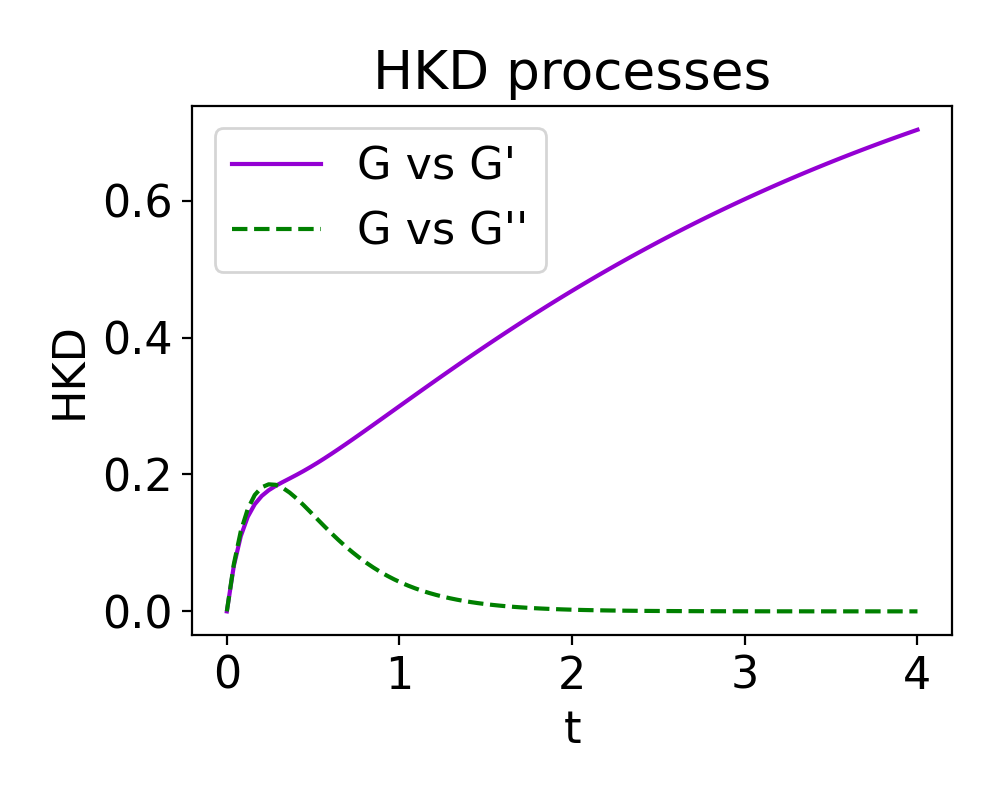}
\caption{Top: representation of three unweighted graphs $G$, $G'$, $G''$. Removing the edge $\{5,6\}$ to $G$ gives $G'$. Removing $\{1,2\}$ to $G$ gives $G''$. Bottom: we plot $t \mapsto D_t((G,G'))$ and $t \mapsto D_t((G,G''))$. }
\label{fig:example_HKD}
\end{figure}

To turn the HKD into a parameter-free notion of distance, \cite{hammond2013graph} define the \textit{Graph Diffusion Distance} as $\max_t D_t((G,G'))$. This has the drawback of comparing different pairs of graphs at different times. To avoid this effect and to take advantage of the multi-scale property of heat kernels mentioned above, our approach consists in using the whole function $t \to D_t((G,G'))$. More precisely, considering a probability distribution $P$ on $\Gn \times \Gn$ and a random pair of graphs $(G,G') \sim P$, we are interested in the stochastic process $\{D_t((G,G')), \ t \in [0,T]\}$ for some $T>0$. That is, the process obtained by evaluating the functions of the family $\FHKD := \{ D_t, \ t \in [0,T] \}$ on a random pair of graphs $(G,G')$. This framework corresponds to the general framework of \emph{empirical processes}.

\subsubsection{Heat Persistence Distance}
In practice, the NC between graphs is not always known. Additionally one may be interested in comparing graphs of different sizes. In these cases, HKD cannot be computed. To circumvent these issues and following ideas from \cite{carriere2020perslay}, we define the Heat Persistence Distance (HPD) by using extended persistence diagrams computed with the Heat Kernel Signature (HKS). These persistence diagrams can be compared with the Bottleneck distance $\db$ without any assumption on graph sizes and node
correspondence.

The HKS was first introduced by \cite{sun2009concise} for the study of shapes. Here we restrict ourselves to the definition of the HKS on graphs of \cite{hu2014stable}.
For a graph $G$ of size $\nnodes$ with vertex set $V = \{ 1, \dots , \nnodes \}$, the HKS at time $t$ is the function $h_t(G) : V \to \R$ such that
\begin{equation*}
h_t(G)(i) = \sum\limits_{k = 1}^\nnodes e^{-t \lambda_k} \phi_k(i)^2, \quad 1 \leq i \leq s.
\end{equation*}
Intuitively, the image of $h_t(G)$ corresponds to the diagonal of the heat kernel $e^{-tL}$. Hence, $h_t(G)(i)$ represents the remaining amount of heat at node $i$ after a diffusion time $t$, when a single unit of heat was placed at node $i$ at time $t=0$.
For each value of $t$, the HKS provides a function on the vertices of a graph, that we use to compute extended persistence diagrams.
Recall that these persistence diagrams encode the upward branches, downward branches, connected components and loops of the graph, when the up/down orientation is given by the HKS. 
The HPD at time $t$ between two graphs $G, G'$ in $\Gun$ is defined by
\begin{equation}
H_t((G,G')) = \underset{\dg}{\max} \   \db(\dg(G,h_t(G)) , \dg(G',h_t(G'))),
\label{eq:HPD_def}
\end{equation}
where the maximum is taken over the four diagram constructions $Ord_0$, $Rel_1$, $Ext_0$ and $Ext_1$.
In our simulations, persistence diagrams are computed by following the approach of \cite{carriere2020perslay} and using the Gudhi library \cite{maria2014gudhi}. 

Similarly to $\FHKD$, we define the family $\FHPD := \{ H_t, \ t \in [0,T] \}$ in order to study the induced stochastic process: $\{ H_t((G,G')), \ t \in [0,T] \}$, for some random pair of graphs $(G,G') \in \Gun \times \Gun$.

The statistical properties of the processes associated with $\FHKD$ and $\FHPD$ are developed in Section~\ref{subsec:stat_prop_hd}. But we first provide in Section~\ref{sec:emp_proc} a more general study of such empirical processes.

\section{General continuous empirical processes}
\label{sec:emp_proc}

In this section, we properly introduce the general framework for continuous empirical processes, then show that uniform boundedness and Lipschitz-continuity implies a functional central limit theorem, as well as a Gaussian approximation. Finally, we derive consequences on the construction of confidence bands and two-sample tests. 

\subsection{Background and definitions}
\label{subsec:emp_proc}

Let $I$ be a compact interval of $\R$ and $\C(I)$ the space of continuous real-valued functions on $I$ endowed with the metric induced by the uniform norm: $\norm{h}_\infty = \sup_{t \in I} \vert h(t) \vert$.
Consider a measurable space $(\bbX, \X)$. For all measures $Q$ on $(\bbX, \X)$ and all measurable functions $g : \bbX \to \R$, we denote the integral of $g$ with respect to $Q$ by $Qg := \int_\bbX g(x) dQ(x)$.
Consider a probability measure $P$ on $(\bbX, \X)$ and $\F := \{f_t, \ t \in I \}$, a family of measurable real-valued functions on $\bbX$ indexed by $I$.
For all $x \in \bbX$, define $f(x)$ as the function $t \to f_t(x)$, and assume that $f(x) \in \C(I)$.
Therefore, given a random variable $X$ with distribution $P$, one can equivalently see $\{ f_t(X), \ t \in I \}$ either as a random process or as $f(X)$ a random variable in $\C(I)$.

Given an i.i.d sample $X_1, \dots, X_N$ drawn under $P$, we are interested in the statistical properties of the mean function $N^{-1} \sum_i f(X_i)$ and its centered and scaled version $N^{-1/2} \left( \sum_i f(X_i) - Pf \right)$.
Equivalently, one can study the empirical processes $\{P_N f_t , \ t \in I \}$ and $\{G_N f_t, \ t \in I \}$, where $P_N = N^{-1} \sum_i \delta_{X_i}$, and $G_N = \sqrt{N} (P_N - P)$. In the following, we see random processes and random functions as the same objects.

When studying the statistical properties of $\F$, one might be interested in a functional version of the central limit theorem. This corresponds to the concept of Donsker families.
\begin{definition}[P-Donsker]
Let $P$ be a probability measure on $(\bbX, \X)$, the family $\F$ is called $P$-Donsker if the process $\{G_{N} f_t, \ t \in I \}$ converges in distribution to the centered Gaussian process $\{\bbG_t, \ t \in I\}$ with covariance function $\kappa$ defined as $\kappa_{s,t} := P f_t f_s -P f_t  P f_s$ for all $t,s \in I$.
\label{def:donsker}
\end{definition}
Here convergence in distribution means weak convergence in the space $\C(I)$. That is, for all continuous bounded functions $h : \C(I) \to \R$, $\lim_{N \to \infty}\expect{}{ h \left( G_N f \right) } =  \expect{}{ h \left( \bbG \right)}$,
where the expectation on the left-hand side is taken over the distribution of the sample $X_1, \dots, X_N$, and the one on the right-hand side is taken over the distribution of the Gaussian process $\bbG$.

Going further into the statistical analysis of $\F$, one might want to assess the speed at which $\{G_{N} f_t, \ t \in I \}$ converges in distribution to $\{ \bbG_t , \ t \in I \}$. This can be done by proving Gaussian approximation results.

\begin{definition}[Gaussian Approximation]
Let $(r_N)_{N \geq 1}$ be a vanishing sequence of positive real numbers. We say that the process $\{G_{N} f_t, \ t \in I \}$ admits a \textit{Gaussian approximation with rate $r_N$}, if for all $\lambda>1$,  there exists a constant $C$ such that for all $N \geq 1$ one can construct on the same probability space both the sample $X_1, \dots, X_N$ and a version $\bbG^{(N)}$ of the Gaussian process $\bbG$ verifying
\begin{equation*}
\proba{}{ \norm{ G_N f - \bbG^{(N)} }_\infty  > C . r_N  }\leq N^{- \lambda}.
\end{equation*}
\label{def:gauss_approx}
\end{definition}

Note that if $\{G_{N} f_t, \ t \in I \}$ admits a Gaussian approximation with rate $r_N$, applying the Borel-Cantelli Lemma would yield $\norm{ G_N f - \bbG^{(N)} }_\infty = O(r_N)$ \textit{almost surely}.

\subsection{Donsker theorem and Gaussian approximation}

Before stating the central theorem of the paper, we introduce two assumptions on $\F$.
\begin{itemize}[align=left]
\item[\lipsch -]  There exits $k>0$ such that for all $x \in \bbX$ the function $t \to f_t(x)$ is $k$-Lipschitz continuous on $I$, meaning that for all $t,s \in I$
\[ \abs{f_t(x) - f_s(x)} \leq k \abs{t-s}. \]
\item[\bound -] $\F$ is uniformly bounded. That is, there exists a constant $M>0$ such that for all $x \in \bbX$ and for all $t \in I$,
\[\abs{f_t(x)} \leq M. \]
\end{itemize}

Remark that assumptions \lipsch and \bound are simple and can easily be checked for most processes. We now show that they are sufficient to obtain a Donsker theorem and a Gaussian approximation result. 

\begin{theorem}
Assume that $\F$ verifies assumptions \lipsch and \bound. \\
Then for any probability measure $P$ on $(\bbX, \X)$, $\F$ is $P$-Donsker and $\{G_{N} f_t, \ t \in I \}$ admits a Gaussian approximation with rate $r_N = N^{-1/7} \log N^{9/14}$.
\label{theo:donsker}
\end{theorem}

This theorem provides a functional central limit theorem as well as some information about the rate of convergence. It allows us to derive in the next section more practical consequences. Namely, it validates the construction of consistent confidence bands and consistent two-sample tests using bootstrap methods. 

The proof of Theorem~\ref{theo:donsker} can be found in Appendix~\ref{sec:proof_donsker}. The Donsker property is proved by standard arguments of tightness. The proof of Gaussian approximation is based on a result from \cite{berthet2006revisiting} requiring technical work and a more complex formalism. Essentially, by using Lipschitz-continuity we control the covering number of $\F$, which is a quantity that indicates the complexity of the family of functions. Even if Theorem~\ref{theo:donsker} is a consequence of \cite{berthet2006revisiting}, working in the special case of continuous processes indexed by a real parameter allows us to present a more accessible result. And although the main interest of the paper is its application to the HKD and HPD processes, we believe that the framework is simpler and that this theorem could easily be applied to other processes.

\subsection{Statistical applications}
\label{subsec:stat_cons}

Here we present how we can construct confidence bands around empirical mean processes and two-sample tests while retrieving statistical guarantees from Theorem~\ref{theo:donsker}.

\subsubsection*{Confidence Band}
Let $c_\alpha := \inf\{ u , \ \proba{}{ \norm{\bbG}_{\infty} > u } \leq \alpha \}$ be the upper $\alpha$-quantile of the maximum of the Gaussian limit process. As a consequence of Theorem~\ref{theo:donsker}, we have
\begin{equation*}
\lim_{N \to \infty} \proba{}{ \forall t, P f_t \in \left[ P_N f_t - \frac{c_\alpha}{\sqrt{N}}, P_N f_t + \frac{c_\alpha}{\sqrt{N}} \right] } \geq 1 - \alpha .
\end{equation*}
Unfortunately, as the distribution of $\bbG$ is unknown, $c_\alpha$ can not be directly computed. Instead, consider a bootstrap sample $\hat{X}_1, \dots, \hat{X}_N$ drawn under $P_N$ and let $\hat{P}_N$ be its empirical probability measure. Consider the process $\{\hat{G}_N f_t , t \in I\}$ where $\hat{G}_N$ is the measure $\sqrt{N} ( \hat{P}_N - P_N )$ and let $\hat{c}_\alpha$ be the upper $\alpha$-quantile of $\|\hat{G}_N f \|_{\infty}$ given the data. Theorem~2.6 in \cite{kosorok2008introduction} ensures that when $\F$ is $P$-Donsker, $\{\hat{G}_N f_t , t \in I\}$ converges weakly to $\bbG$, given the data. Hence $\hat{c}_\alpha$ provide an approximation of $c_\alpha$, and we can use $\hat{c}_\alpha$ to design a consistent confidence band around $P_N f$:
\begin{equation*}
\lim_{N \to \infty} \proba{}{  \forall t, P f_t \in \left[ P_N f_t - \frac{\hat{c}_\alpha}{\sqrt{N}}, P_N f_t + \frac{\hat{c}_\alpha}{\sqrt{N}} \right]}) \geq 1 - \alpha .
\end{equation*}
Note that $\hat{c}_\alpha$ can be estimated with Monte-Carlo simulations, by drawing as many bootstrap samples as we want.

\subsubsection*{Two Sample Test}
 Consider the following setup. Let $P$ and $Q$ be two probability distributions on $\bbX$, and assume we are given two independent iid samples $(X_1, \dots , X_M)$ and $(Y_1, \dots Y_N)$, drawn under $P$ and $Q$,  respectively. We denote by $P_M$ and $Q_N$ the empirical measures. We would like to test the null hypothesis $H_0 : P=Q$ against the alternatives $H_1 : P \neq Q$, by using the family $\F$, assuming that it is Donsker with respect to both $P$ and $Q$. We follow the approach described in Section~3.7 of \cite{van1996weak}, and consider the following test statistic:
\begin{equation*}
 D_{M,N} := \sqrt{\frac{MN}{M+N}} \norm{P_M f - Q_N f}_{\infty}.
\end{equation*}
The strategy is to define a data-dependent threshold $\hat{c}_{M,N}(\alpha)$ and reject the null hypothesis whenever $D_{M,N}>\hat{c}_{M,N}(\alpha)$. Consider the pooled data $(Z_1, \dots Z_{M+N}) = (X_1 , \dots, X_M, Y_1, \dots, Y_N)$, and its empirical measure $H_{N+M}$. Let $(\hat{Z}_1, \dots \hat{Z}_{M+N})$ be a bootstrap sample drawn from $H_{M+N}$ and consider the bootstrap empirical measures
\begin{equation*}
\hat{P}_M = \frac{1}{M} \sum\limits_{i = 1}^{M} \delta_{\hat{Z}_i} \qquad \textnormal{and} \qquad \hat{Q}_N = \frac{1}{N} \sum\limits_{i = 1}^{N} \delta_{\hat{Z}_{M+i}}.
\end{equation*}
We can define
\begin{equation*}
\hat{D}_{M,N} := \sqrt{\frac{MN}{M+N}} \norm{\hat{P}_M f - \hat{Q}_N f}_{\infty},
\end{equation*}
as well as
\begin{equation*}
\hat{c}_{M,N}(\alpha) = \inf \left\{ t, \ \proba{}{ \hat{D}_{M,N}  > t \  \Big\vert \ Z_1, \dots, Z_{N+M} } \leq \alpha \right\},
\end{equation*}
for $\alpha \in (0,1)$.
Note that $\hat{c}_{M,N}(\alpha)$ can be estimated with Monte-Carlo simulations. Using $\hat{c}_{M,N}(\alpha)$ as the threshold to accept or reject $H_0$ leads to a consistent test.

\begin{theorem}[\cite{van1996weak}, Section 3.7.2]
Assume that $\F$ is Donsker with respect to both $P$ and $Q$ and that $\norm{Pf}_\infty$ and $\norm{Qf}_\infty$ are finite. Furthermore assume that $M / (M+N) \to \lambda \in (0,1)$. Then the test that rejects $H_0$ whenever $D_{M,N} > \hat{c}_{M,N}(\alpha)$ is consistent, in the sense that the asymptotic level is $\alpha$ and under any alternative verifying $\norm{Pf-Qf}_{\infty} > 0$, $\proba{}{D_{M,N} > \hat{c}_{M,N}(\alpha)} \to 1$.
\label{theo:test_consistency}
\end{theorem}

This provides guarantees of the two-sample test when both sample sizes grow to infinity. Intuitively, under $H_0$, as both the process defining $D_{M,N}$ and the one defining $\hat{D}_{M,N}$ given $Z_1, \dots, Z_{N+M}$ converge to the same Gaussian process, we can sample from the distribution of the latter (randomness coming only from resampling) to estimate the quantile of the former.

\section{Back to the heat distance processes}
\label{sec:result_hd}

Using the results of the previous section, we are able to show that the HKD and HPD processes admit a functional version of the central limit theorem, as well as a Gaussian approximation. This is essentially based on the fact that they are uniformly bounded Lipschitz-continuous processes. From these results, we propose a procedure to construct consistent confidence bands. We also detail the construction of a two-sample test for samples of pairs of graphs based on either the HPD processes or the HKD processes when the graph sizes and NC allow their computation. We explicit these constructions and state their consistency results.  \medskip

\subsection{Statistical properties}
\label{subsec:stat_prop_hd}

Let $P$ be a probability distribution on $\Gn \times \Gn$. Let $T$ be a positive real number and recall that $\FHKD := \{ D_t, \ t \in \segm \}$, with $D_t$ defined in (\ref{eq:HKD_def}). We will study the centered and rescaled empirical process $\{G_N D_t, \ t \in I \} = \{\sqrt{N} (P_N - P) D_t, \ t \in I \} $, where $P_N$ is the empirical measure $N^{-1} \sum_i \delta_{(G_i, {G'}_i)}$ associated to a $N$-sample $((G_1,{G'}_1), \dots, (G_N,{G'}_N))$ drawn under $P$.
We now state the statistical results concerning the HKD process.

\begin{theorem}
For any probability distribution $P$ on $\Gn \times \Gn$, the family $\FHKD$ is $P$-Donsker. That is, the process $\{\sqrt{N} (P_N - P)D_t , \ t \in \segm \}$ converges weakly to $\bbG$ in $\C(\segm)$, where $\bbG = \{ \bbG_t , \ t \in \segm\}$ is a zero mean Gaussian process with covariance function $\kappa(t,s) = P(D_t D_s) - PD_t . PD_s$. \\
Moreover, it admits a Gaussian approximation with rate $r_N = N^{-1/7} \log N^{9/14}$.
\label{theo:HKDdonsker}
\end{theorem}

We have a similar result for HPD processes. Recall that for HPD processes we work in $\Gun$. Let $P$ be a probability distribution on $\Gun \times \Gun$. Let $T$ be a positive real number and recall that $\FHPD := \{ H_t, \ t \in \segm \}$, with $H_t$ defined in (\ref{eq:HPD_def}). Again, we study the empirical process $\{G_N H_t, \ t \in I \} = \{\sqrt{N} (P_N - P) H_t, \ t \in I \} $, where $P_N$ is the empirical measure $N^{-1} \sum_i \delta_{(G_i, {G'}_i)}$ associated to a $N$-sample $((G_1,{G'}_1), \dots, (G_N,{G'}_N))$ drawn under $P$.

\begin{theorem}
For any probability distribution $P$ on $\Gun \times \Gun$, the family $\FHPD$ is $P$-Donsker. Thus, the process $\{\sqrt{N} (P_N - P)H_t , \ t \in \segm \}$ converges weakly to $\bbG$ in $\C(\segm)$, where $\bbG = \{ \bbG_t , \ t \in \segm\}$ is a zero mean Gaussian process with covariance function $\kappa(t,s) = P(H_t H_s) - PH_t . PH_s$. \\
Moreover, it admits a Gaussian approximation with rate $r_N = N^{-1/7} \log N^{9/14}$.
\label{theo:HPDdonsker}
\end{theorem}

These theorems can be seen as functional central limit theorems for the HKD and HPD processes. They validate bootstrap methods for the construction of consistent confidence bands and consistent two-sample tests. Gaussian approximations strengthen these results and provide information about the speed of convergence to Gaussian processes. Note that this rate is independent of the graph size $\nnodes$. This could indicate that even in the finite sample case with large graphs, methods based on the Donsker property could work well.  

The proof of Theorem~\ref{theo:HKDdonsker} and Theorem~\ref{theo:HPDdonsker} can be found in Appendix~\ref{sec:HDdonsker_proof}. We start by showing that under the hypothesis of bounded weights on the edges of the graphs the HKD and HPD process are uniformly bounded and Lipschitz-continuous, then the proofs are direct applications of Theorem~\ref{theo:donsker}. 

\subsection{Confidence band and consistency}

As explained in Section~\ref{subsec:stat_cons}, Donsker results allow for the construction of consistent confidence bands using bootstrap methods. In Algorithm~\ref{alg:conf_band}, we detail this construction based on the HKD and HPD processes on a sample of pairs of graphs. As the construction is similar in the HKD or HPD case, we only present the HKD case.

\begin{algorithm}
\SetKwInOut{Input}{Input}
\caption{Confidence Band}
\label{alg:conf_band}
\KwData{a sample of pairs of graphs $((G_1, G'_1) , \dots (G_N, G'_N))$}
\Input{an integer $B$ and a level $\alpha$.}
\BlankLine
\For{$i \in \{1, \dots, N\}$}{
	compute $d_i : t \mapsto D_t((G_i, G'_i))$\;
}
compute $\overline{d}_N = N^{-1} \sum_i d_i$\;
\For{$b \in \{1, \dots, B\}$}{
	draw with replacement a bootstrap sample $(\hat{d}^{(b)}_1, \dots \hat{d}^{(b)}_N)$ out of $(d_1, \dots d_N)$\;
	compute $\tilde{d}^{(b)}_N = N^{-1} \sum_i \hat{d}^{(b)}_i$\;
	compute $t^{(b)} = \sqrt{N} \| \tilde{d}^{(b)}_N - \overline{d}_N  \|_\infty$\;
}
compute $\tilde{c}_\alpha$ the empirical upper $\alpha$-quantile of $t^{(1)}, \dots, t^{(B)}$\;
\BlankLine
\KwResult{ the confidence band $\left[ \overline{d}_N(t) - \frac{\tilde{c}_\alpha}{\sqrt{N}}, \overline{d}_N(t) + \frac{\tilde{c}_\alpha}{\sqrt{N}} \right]$ for all $t$.}

\end{algorithm}

By choosing the bootstrap sample size $B$ large enough, $\tilde{c}_\alpha$ can approximate as closely as we want the upper $\alpha$-quantile $\hat{c}_\alpha$ of the bootstrap test statistic $t^{(1)}$ given the data. And according to Section~\ref{subsec:stat_cons} we know that
\begin{equation*}
\lim_{N \to \infty} \proba{}{  \forall t, \expect{P}{D_{t}((G,G'))} \in \left[ \overline{d}_N(t) - \frac{\hat{c}_\alpha}{\sqrt{N}}, \overline{d}_N(t) + \frac{\hat{c}_\alpha}{\sqrt{N}} \right]}) \geq 1 - \alpha .
\label{eq:consistency_cb}
\end{equation*}
This means that for large enough $B$, Algorithm~\ref{alg:conf_band} produces a consistent confidence band.

\subsection{Two-sample test and consistency}

We now detail how to construct a two-sample test for sample of pairs of graphs in Algorithm~\ref{alg:two_sample_test}. As previously, we present the HKD case, as the HPD case is similar.

\begin{algorithm}
\SetKwInOut{Input}{Input}
\caption{Two-sample Test}
\label{alg:two_sample_test}
\KwData{two samples of pairs of graphs \\ $((G_{1,1}, G'_{1,1}) , \dots (G_{1,N}, G'_{1,N}))$ and $((G_{2,1}, G'_{2,1}) , \dots (G_{2,M}, G'_{2,M}))$}
\Input{an integer $B$ and a level $\alpha$.}
\BlankLine
\For{$i \in \{1, \dots, N\}$}{
	compute $d_{1,i} : t \mapsto D_t((G_{1,i}, G'_{1,i}))$\;
}
\For{$i \in \{1, \dots, M\}$}{
	compute $d_{2,i} : t \mapsto D_t((G_{2,i}, G'_{2,i}))$\;
}
compute the empirical means $\overline{d}_1 = N^{-1} \sum_i d_{1,i}$ and $\overline{d}_2 = M^{-1} \sum_i d_{2,i}$\;
compute the test statistic $T = \sqrt{NM (N+M)^{-1}} \|d_1 - d_2\|_\infty$\;
\For{$b \in \{1, \dots, B\}$}{
	draw with replacement $(\hat{d}^{(b)}_{1,1}, \dots \hat{d}^{(b)}_{1,N})$ and $(\hat{d}^{(b)}_{2,1}, \dots \hat{d}^{(b)}_{2,M})$  out of $(d_{1,1}, \dots d_{1,N}, d_{2,1}, \dots d_{2,M})$\;
	compute $\tilde{d}^{(b)}_{1} = N^{-1} \sum_i \hat{d}^{(b)}_{1,i}$ and $\tilde{d}^{(b)}_{2} = M^{-1} \sum_i \hat{d}^{(b)}_{2,i}$ \;
	compute $T^{(b)} = \sqrt{NM (N+M)^{-1}} \|\tilde{d}^{(b)}_{1} - \tilde{d}^{(b)}_{2}\|_\infty$\;
}
compute $\tilde{c}_\alpha$ the empirical upper $\alpha$-quantile of $T^{(1)}, \dots, T^{(B)}$\;
\BlankLine
\KwResult{ Reject the null hypothesis if  $T > \tilde{c}_\alpha$.}

\end{algorithm}

Once again, for a large enough $B$, $\tilde{c}_\alpha$ closely approximate the the upper $\alpha$-quantile of $T^{(1)}$ given the data. Recall that from Theorem~\ref{theo:test_consistency}, using $\hat{c}_\alpha$ as the rejection threshold yields to a control of the asymptotic level and power. That is 
\begin{equation*}
\lim \proba{H_0}{T > \hat{c}_\alpha} \leq \alpha,
\end{equation*}
and under any alternative verifying $\norm{P_1 D_{\cdot} - P_2 D_{\cdot}}_{\infty} > 0$ where $P_1$ and $P_2$ are the distributions that generated the two samples, 
\begin{equation*}
\lim \proba{H_1}{T > \hat{c}_\alpha} = 1.
\end{equation*}
The limits are for both $N$ and $M$ going to infinity under the condition that $N/(N+M) \to  \lambda$, for a $\lambda$ in $(0,1)$.
Hence, for large enough $B$, Algorithm~\ref{alg:two_sample_test} produces a asymptotically valid two-sample test for samples of pairs of graphs.

\section{Experiments}
\label{sec:experiments}

We illustrate the construction of confidence bands and two-sample tests on synthetic data sets of pairs of graphs. For that, we consider different random graph models and combine them to create independent pairs of graphs.

\subsection{Random graph models} 

In this section, we present the models generating the random graphs, namely the Erd\H{o}s-Rényi model \cite{erdos1960evolution}, the stochastic block model \cite{holland1983stochastic}, the geometric model \cite{penrose2003random} and the Watts-Strogatz model \cite{watts1998collective}. For each, we specify the parameters used in our simulations.

\subsubsection*{Erd\H{o}s-Rényi model (ER)} This model generates random graphs where each edge appears with probability $p$, independently from all the others. It requires two parameters: $n$ the graph size and $p$ the edge probability. Because of the independence and their homogeneity, ER graphs are considered to have no structure.

In our simulations, we take $n=50$ and $p=0.5$. Weights may be added by assigning a uniform weight between 0 and 2 to each existing edge, independently from all the others.

\subsubsection*{Stochastic block model (SBM)}  This model is a generalization of the ER model that introduces a block structure. The $n$ nodes are clustered in $K$ groups $C_1, \dots , C_K$, of respective sizes $n_1, \dots, n_K$. Edges appear independently from the others, but with a probability depending on the groups: edge $\{i,j\}$ appears with probability $p_{k,l}$ when $i \in C_k$ and $j \in C_l$.

In our simulations, we take $K=2$, $n_1 = n_2 = 25$, and $p_{1,1} = p_{2,2} = 0.75$, $p_{1,2} = p_{2,1} = 0.25$. So graphs are composed of two dense clusters, with few edges between them. Similarly to the ER model, we may add random weights following the uniform distribution between 0 and 2.

\subsubsection*{Geometric model (GM)}  Given a compact domain $U$ of $\R^d$ for some $d$, a graph is generated from the GM by drawing $n$ points uniformly on $U$ and creating an edge between two points if their Euclidean distance is smaller than a given threshold. Here we choose a slight variation of this model by considering a number $p \in [0,1]$ and creating the edges corresponding to the $\lfloor p \binom{n}{2} \rfloor $ pairs of points with the smallest euclidean distances.  

In our simulations, the compact domain $U$ is either $A_\epsilon $ the annulus in $\R^2$ with outer radius 1 and inner radius $\epsilon > 0$ or $A_0$ the unit disk. We either fix $n=50$ or for each graph, $n$ is drawn from a Poisson distribution of parameter 50. We take $p=0.5$. To obtain weighted graphs, we may assign the weight $e^{-2d}$ to an edge, where $d$ is the distance between the two points forming the edge.

\subsubsection*{Watts-Strogatz model (WS)} This model is famous for its so-called \emph{small-world} property, meaning that any two nodes in the graph are usually not far apart in graph distance. Random graphs under this model are constructed as follows. First a ring of $n$ nodes is created where each node is connected to its $k$ nearest neighbors. Then, for each edge, with probability $p$ independently of all others edges, one of its terminal nodes is replaced by another node chosen uniformly at random. 

In our simulations, we take $n=50$, $k=25$ and $p=0.5$. Weights may be added by assigning a uniform weight between 0 and 2 to each existing edge, independently from all the others.
\medskip

We combine these models to generate pairs of independent graphs on which we can compute HKD and/or HPD processes. We consider the pairs of independent ER graphs (ER-ER) and the pairs containing one ER graph and one SBM graph (ER-SBM). For these distributions, the groups' composition is known and nodes are treated independently among groups. Thus, we can consider that we know an NC between graphs. As a result, we can compute both HKD and HPD processes. 
Similarly, we consider (ER-ER) and (ER-SW) pairs and compare them with HKD processes (as nodes are exchangeable). 
We also consider pairs of independent geometric random graphs: Disk-Disk and Disk-Annulus. However, as nodes in these models correspond to random points, there is no default NC. Hence, only HPD processes are computed.

\subsection{Simulation results}

In this section, we present the simulation results concerning the construction of confidence bands and the performances of the two-sample tests.

\subsubsection{Confidence bands}

In this section, we compute confidence bands under the different models of pairs of graphs defined above. For each model, we draw a sample $(G_1, G'_1), \dots, (G_N, G'_N)$ with $N=100$. We compute the mean process, that is $t \to N^{-1} \sum_i D_t((G_i,G'_i))$ or $t \to N^{-1} \sum_i H_t((G_i,G'_i))$ and compute a confidence band of level 99\% around this empirical mean using the bootstrap method detailed in Algorithm~\ref{alg:conf_band}. Computations are done with 1000 bootstrap samples. Results are shown in Figure~\ref{fig:cb_HKD_ERvsSBM}, \ref{fig:cb_HPD_ERvsSBM} and \ref{fig:cb_HPD_DiskvsAnnulus}, where solid lines represent empirical means and transparent areas represent confidence bands.

Remark that confidence bands around HKD processes (Figure~\ref{fig:cb_HKD_ERvsSBM}) seem to be narrower than those around HPD processes (Figure~\ref{fig:cb_HPD_ERvsSBM}), relatively to the height of the curves. 
Therefore, users should rather use HKD processes whenever NC's are available. Nonetheless, the versatility of HPD processes does not totally reduce their efficiency. As Figure~\ref{fig:cb_HPD_DiskvsAnnulus} indicates, HPD empirical means seem to be able to discriminate between the different distributions. These observations will be confirmed in the next section, where the performances of the two-sample tests are investigated.

\begin{figure}
\centering
\begin{subfigure}{.45\textwidth}
	\includegraphics[width=\linewidth]{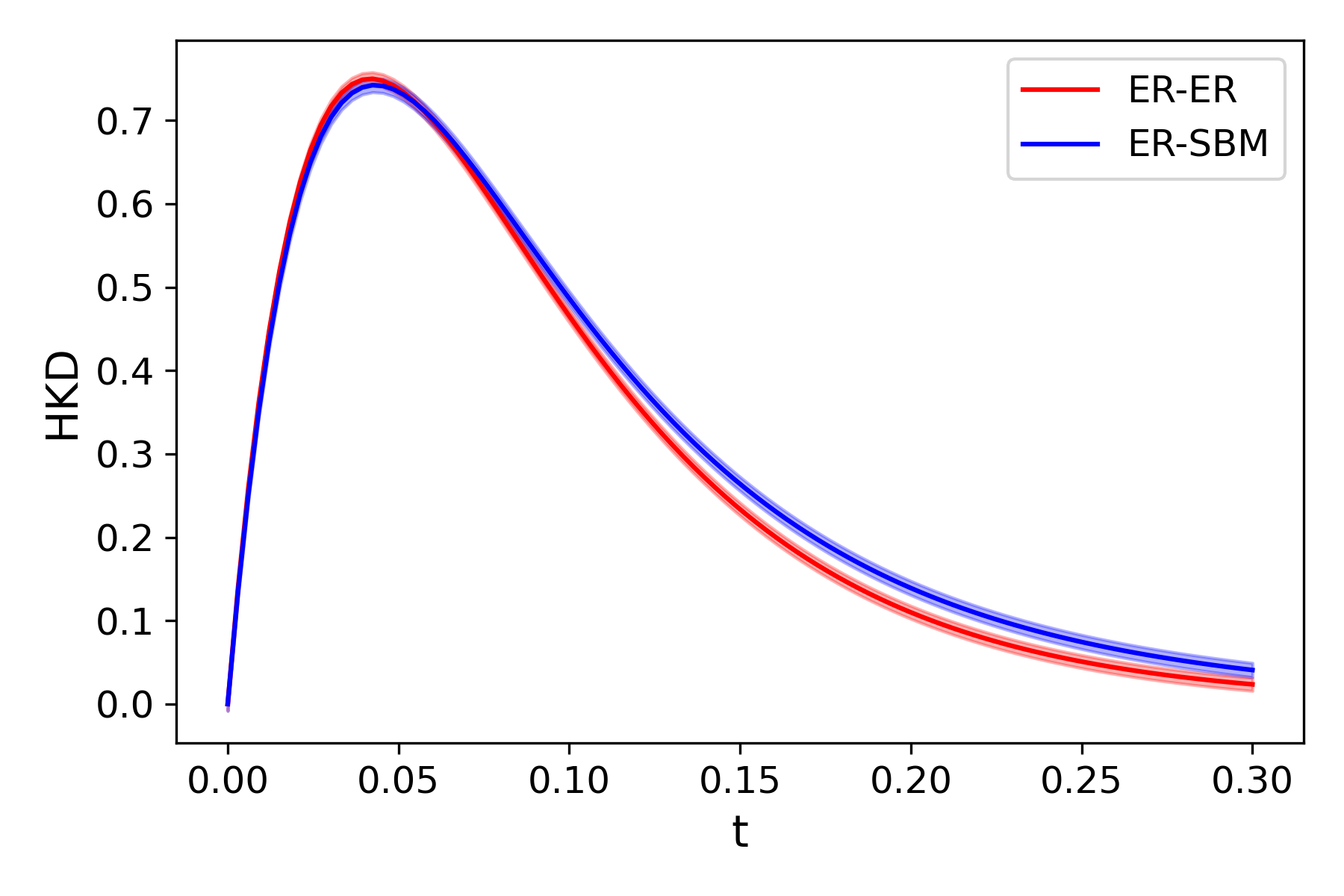}
	\caption{Unweighted.}
	\label{subfig:cb_HKD_ERvsSBM}
\end{subfigure}
\begin{subfigure}{.45\textwidth}
	\includegraphics[width=\linewidth]{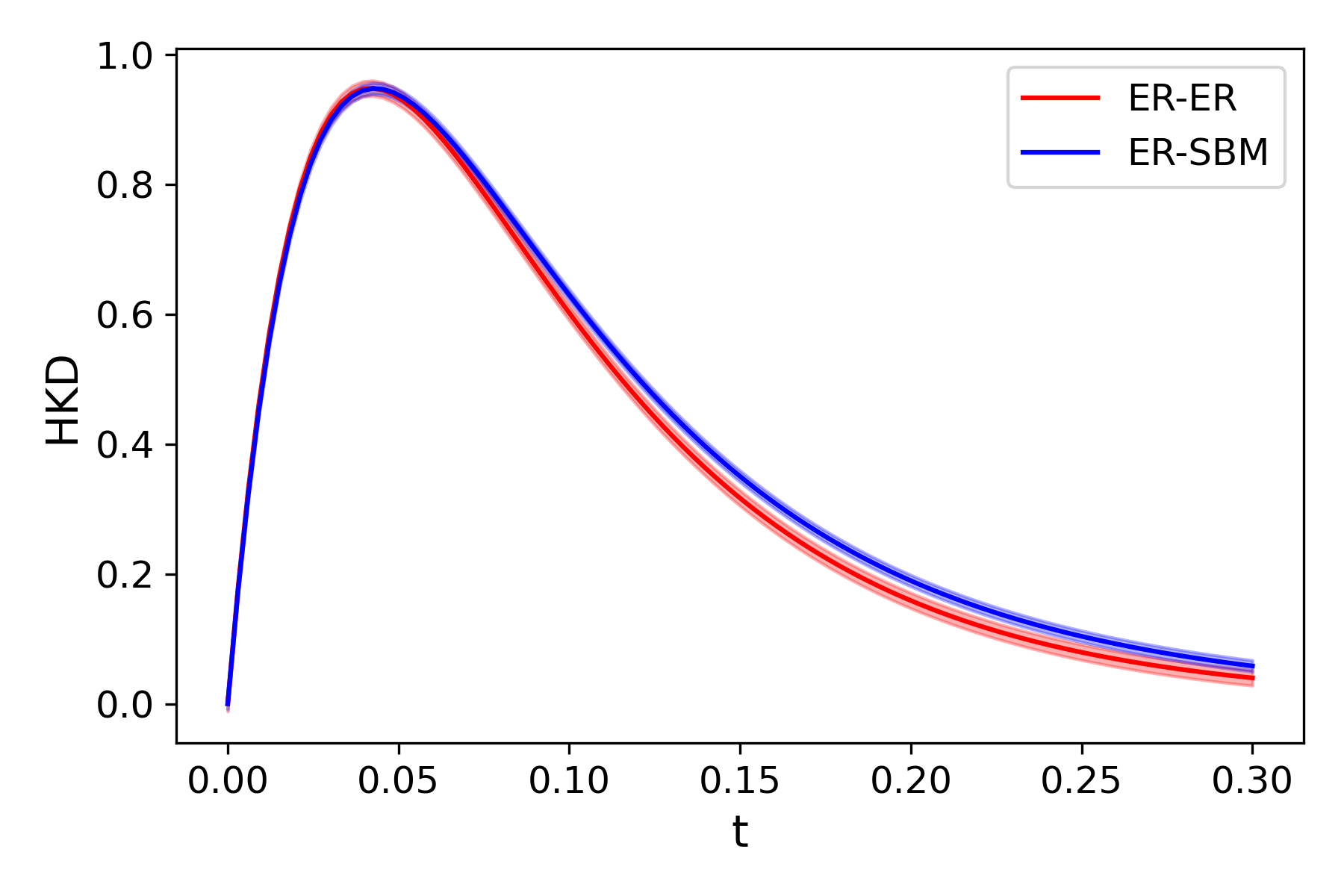}
	\caption{Weighted.}
	\label{subfig:cb_HKD_ERvsSBM_w}
\end{subfigure}
\caption{Confidence band around the mean HKD processes with ER-ER (red) and ER-SBM (blue) distributions.}
\label{fig:cb_HKD_ERvsSBM}
\end{figure}

\begin{figure}
\centering
\includegraphics[width=0.45\linewidth]{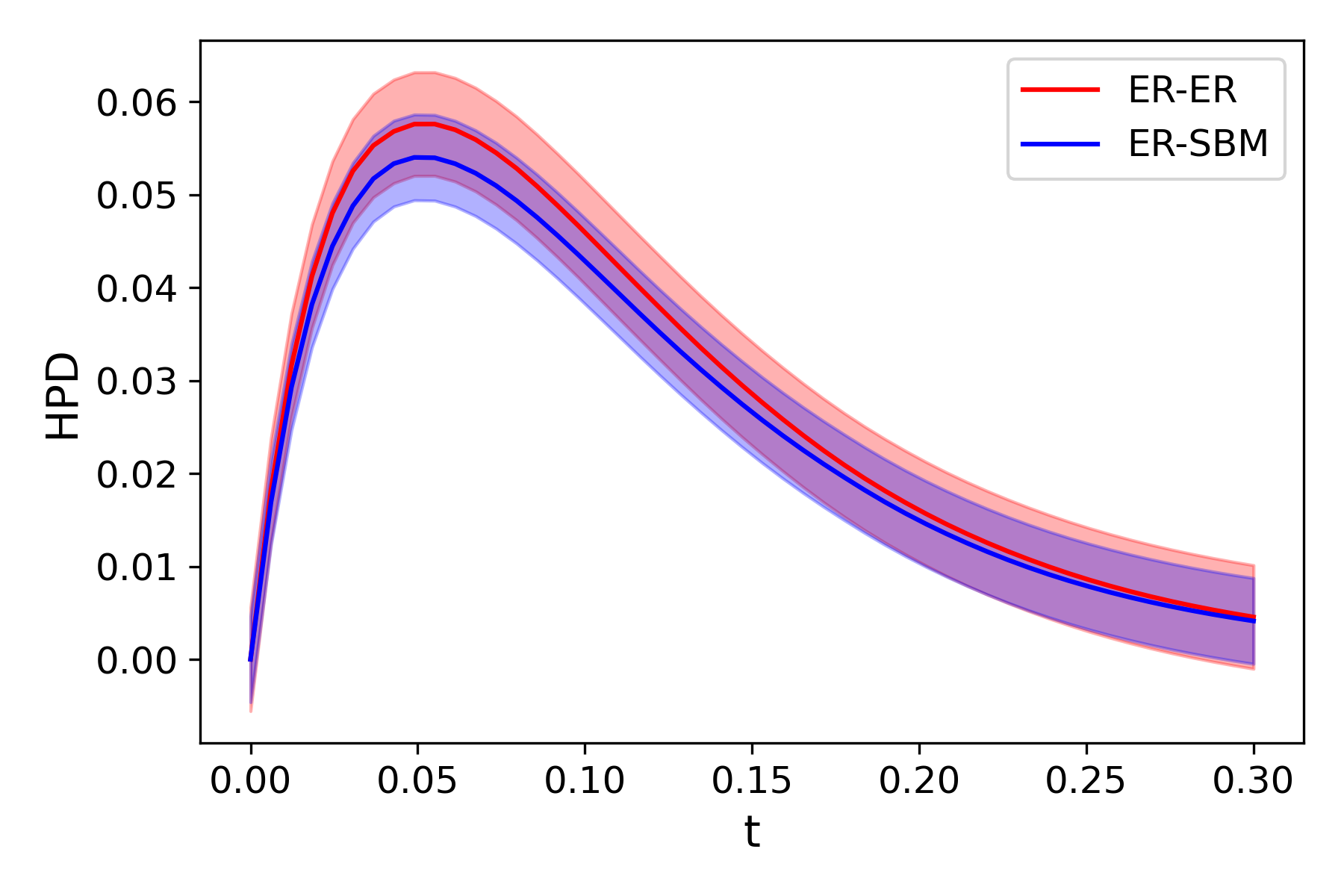}
\caption{Confidence band around the mean HPD processes with unweighted ER-ER (red) and ER-SBM (blue) distributions.}
\label{fig:cb_HPD_ERvsSBM}
\end{figure}

\begin{figure}
\centering
\begin{subfigure}{.32\textwidth}
	\includegraphics[width=\linewidth]{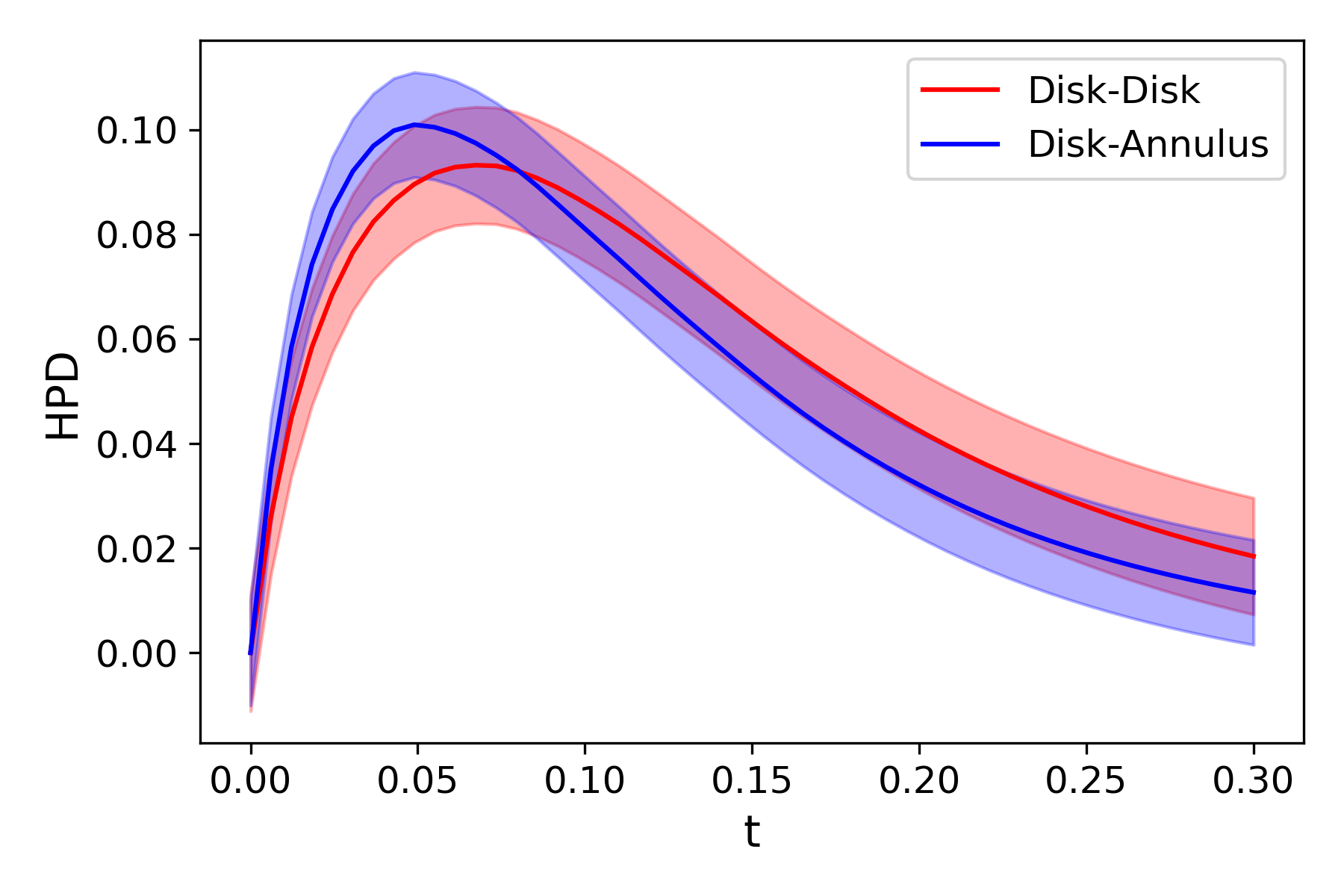}
	\caption{Unweighted, \\ fixed size.}
	\label{subfig:cb_HPD_DiskvsAnnulus}
\end{subfigure}
\begin{subfigure}{.32\textwidth}
	\includegraphics[width=\linewidth]{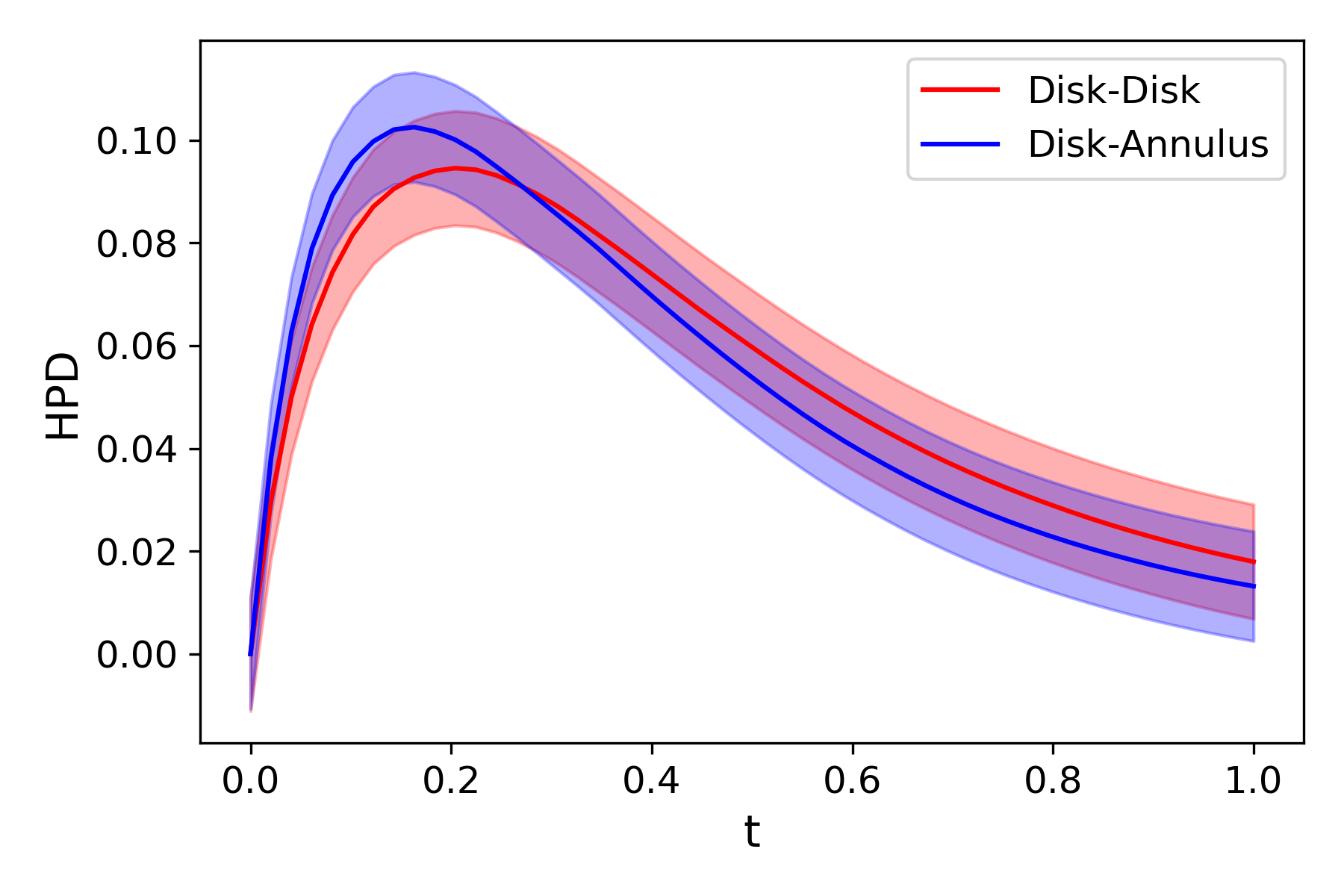}
	\caption{Weighted, \\ fixed size.}
	\label{subfig:cb_HPD_DiskvsAnnulus_w}
\end{subfigure}
\begin{subfigure}{.32\textwidth}
	\includegraphics[width=\linewidth]{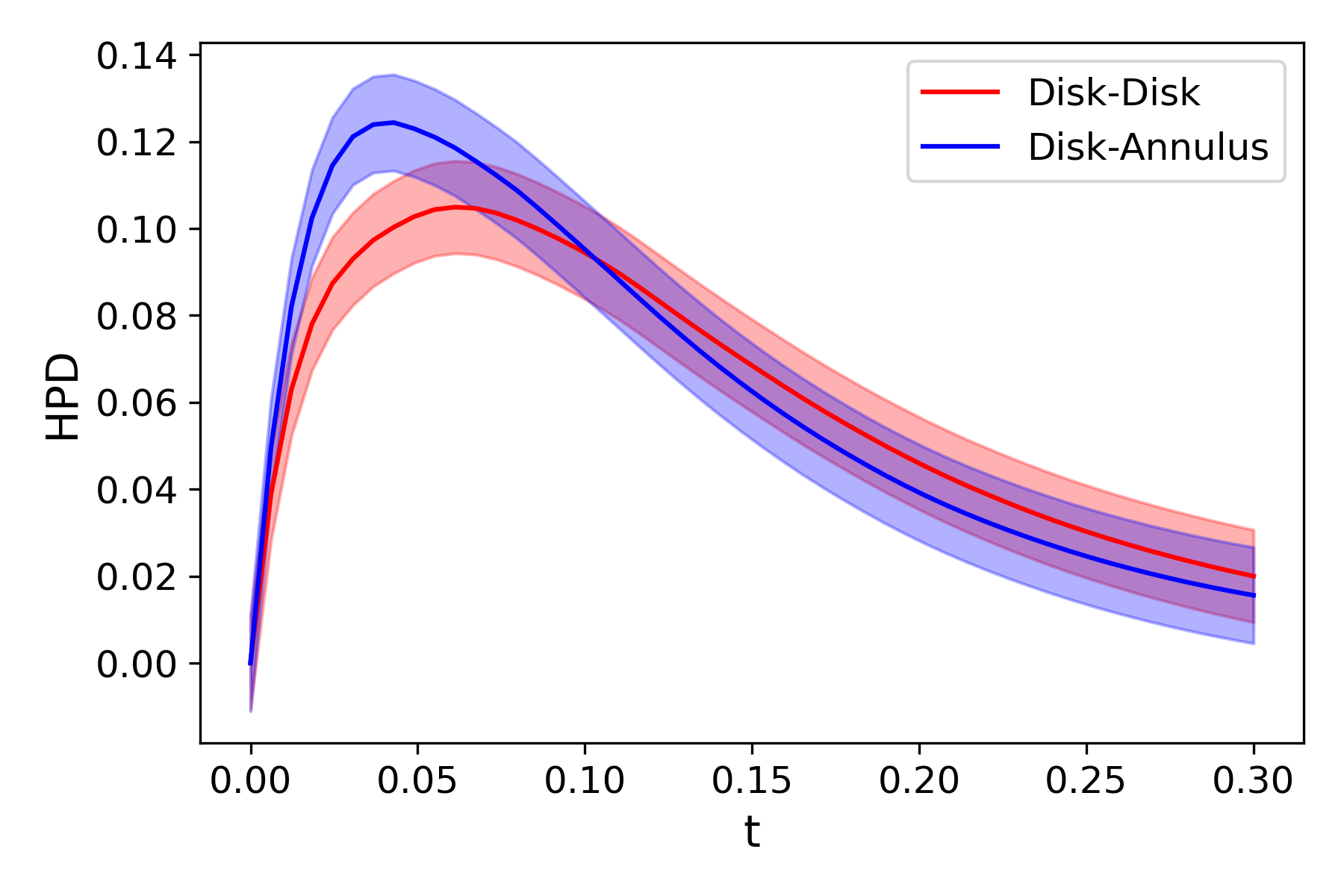}
	\caption{Unweighted, \\ random size.}
	\label{subfig:cb_HPD_DiskvsAnnulus_rdm}
\end{subfigure}
\caption{Confidence band around the mean HPD processes with Disk-Disk (red) and Disk-Annulus (blue) distributions.}
\label{fig:cb_HPD_DiskvsAnnulus}
\end{figure}

\subsubsection{Two-sample tests}

\paragraph*{Levels and Powers.} Simulations are run to evaluate the performances of the two-sample tests using HKD and HPD processes as detailed in Algorithm~\ref{alg:two_sample_test}, see Figure~\ref{fig:test_HKD}, Figure~\ref{fig:test_HPD} and Figure~\ref{fig:test_HPD_rdm_size}. In all tests, the desired level is set to 0.05, and computations are done with 1000 bootstrap samples. Figures~\ref{subfig:test_HKD_typeI} and \ref{subfig:test_HPD_typeI} illustrate that the asymptotic levels of the tests correspond to the set level. On the other hand, Figures~\ref{subfig:test_HKD_power}, \ref{subfig:test_HPD_power} and \ref{fig:test_HPD_rdm_size} illustrate that the powers tend to 1 when sample sizes increase, indicating that the tests manage to distinguish between the different distributions. Moreover, Figure~\ref{subfig:test_HPD_power} and \ref{fig:test_HPD_rdm_size} show the good performances of the HPD-based test even when NC is unknown and when graphs have different sizes. 

\begin{figure}
\centering
\begin{subfigure}{.45\textwidth}
	\includegraphics[width=\linewidth]{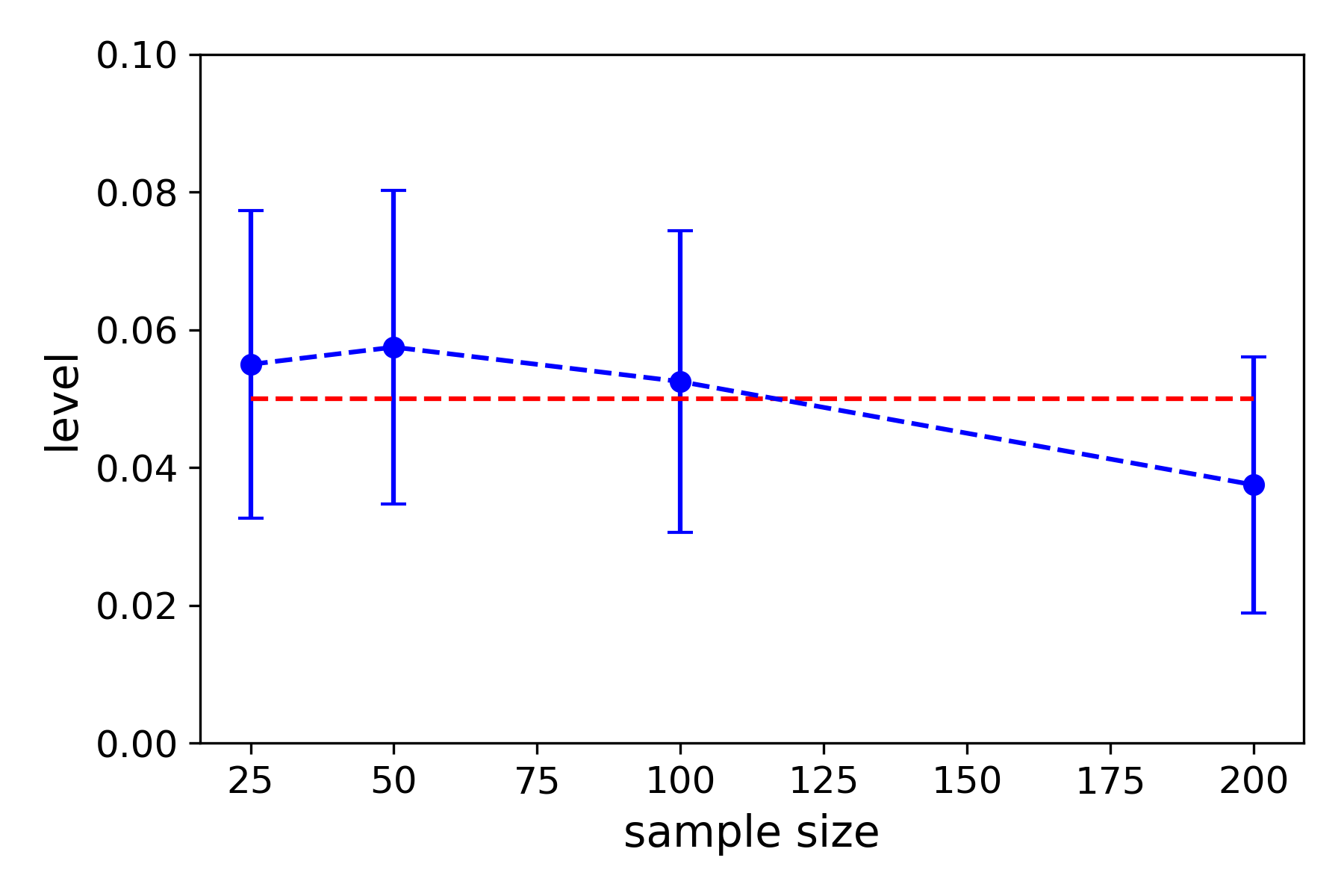}
	\caption{Levels estimations.}
	\label{subfig:test_HKD_typeI}
\end{subfigure}
\begin{subfigure}{.45\textwidth}
	\includegraphics[width=\linewidth]{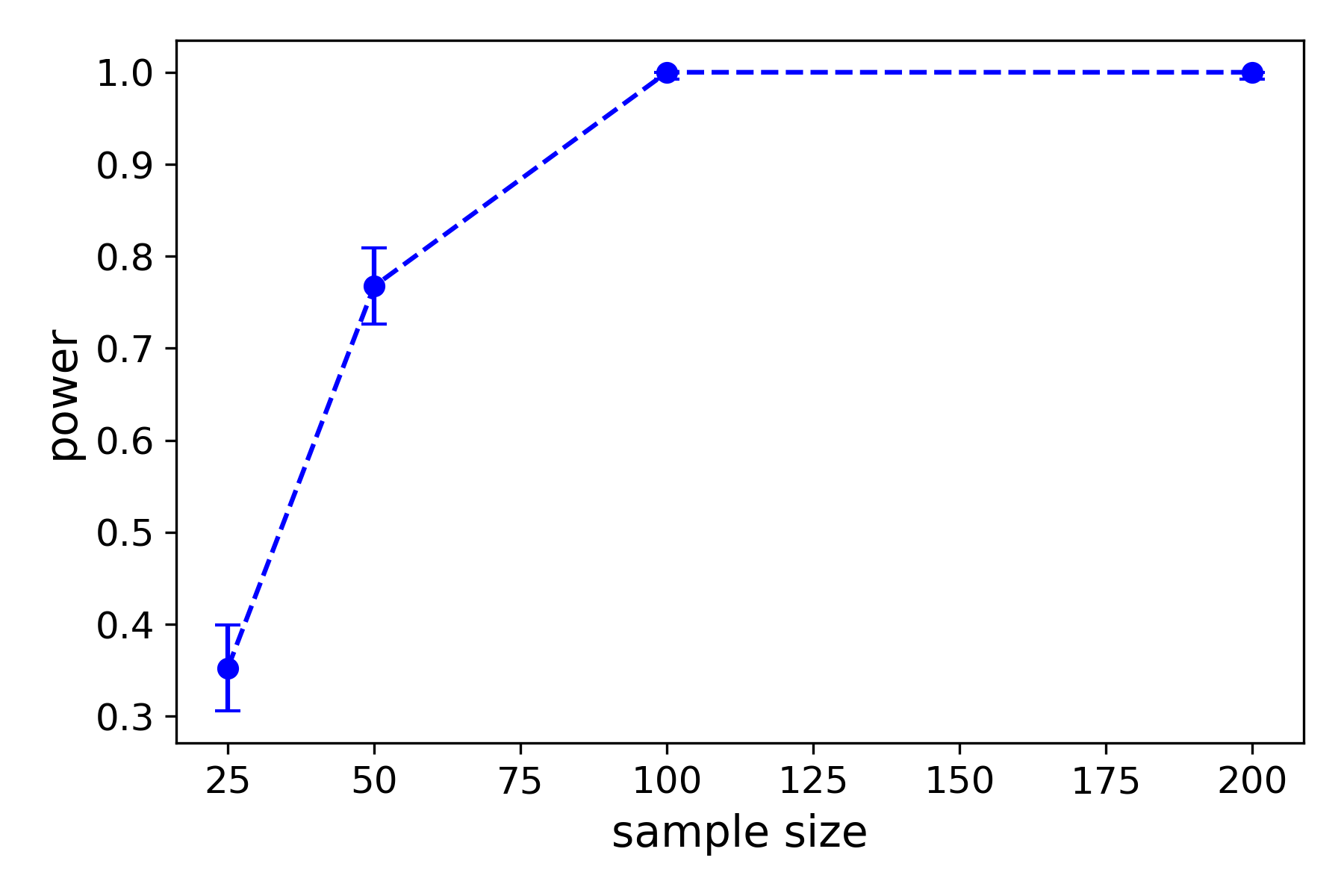}
	\caption{Powers estimations.}
	\label{subfig:test_HKD_power}
\end{subfigure}
\caption{Performances of the two-sample test using HKD processes. The level estimations are made with weighted ER-ER distributions, while the power estimations are made with weighted ER-ER and ER-SBM distributions. The samples size varies in [25,50,100,200]. Estimations are done by repeating 400 independent tests. Vertical lines represent $95\%$-confidence intervals of the estimations.}
\label{fig:test_HKD}
\end{figure}

\begin{figure}
\centering
\begin{subfigure}{.45\textwidth}
	\includegraphics[width=\linewidth]{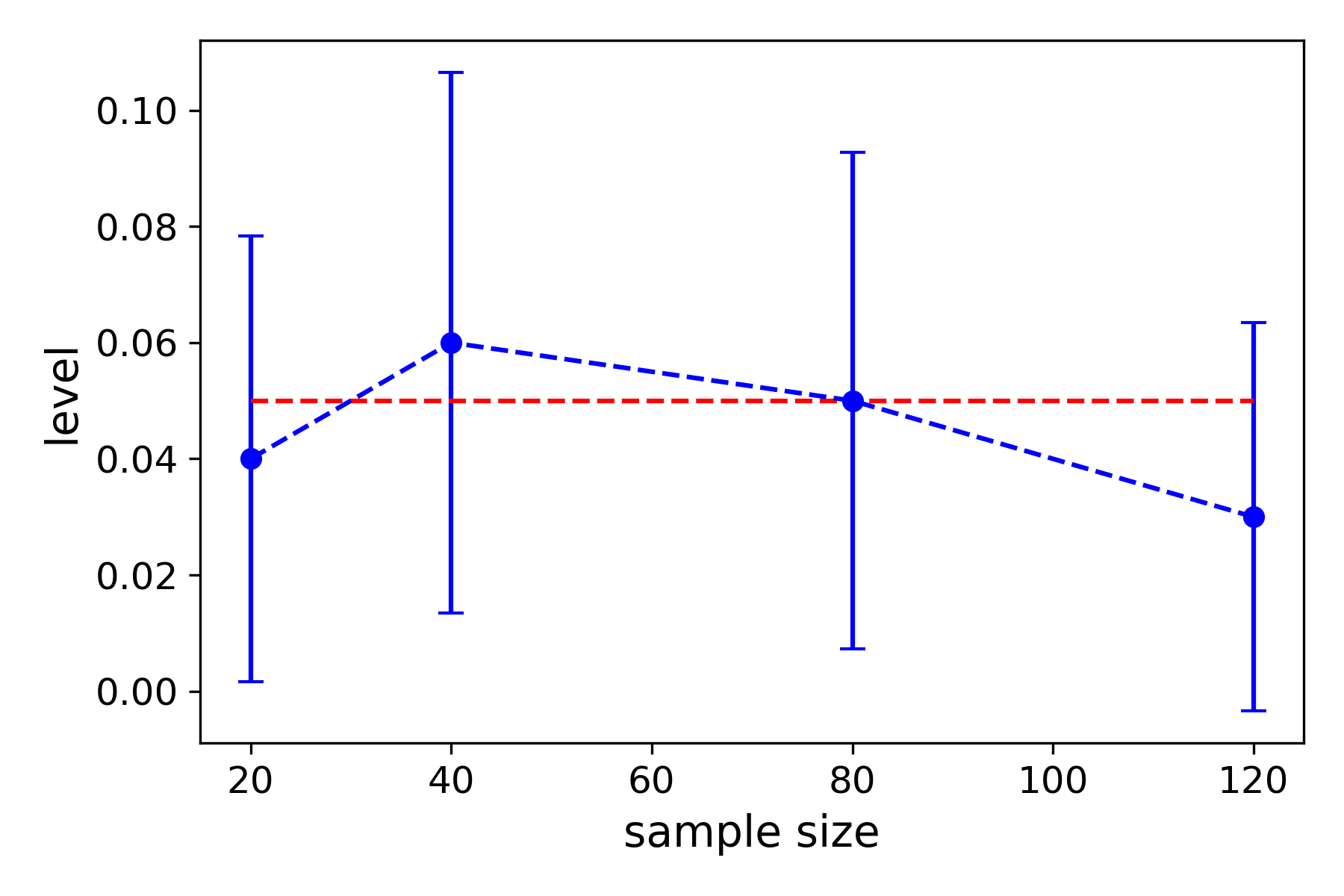}
	\caption{Levels estimations.}
	\label{subfig:test_HPD_typeI}
\end{subfigure}
\begin{subfigure}{.45\textwidth}
	\includegraphics[width=\linewidth]{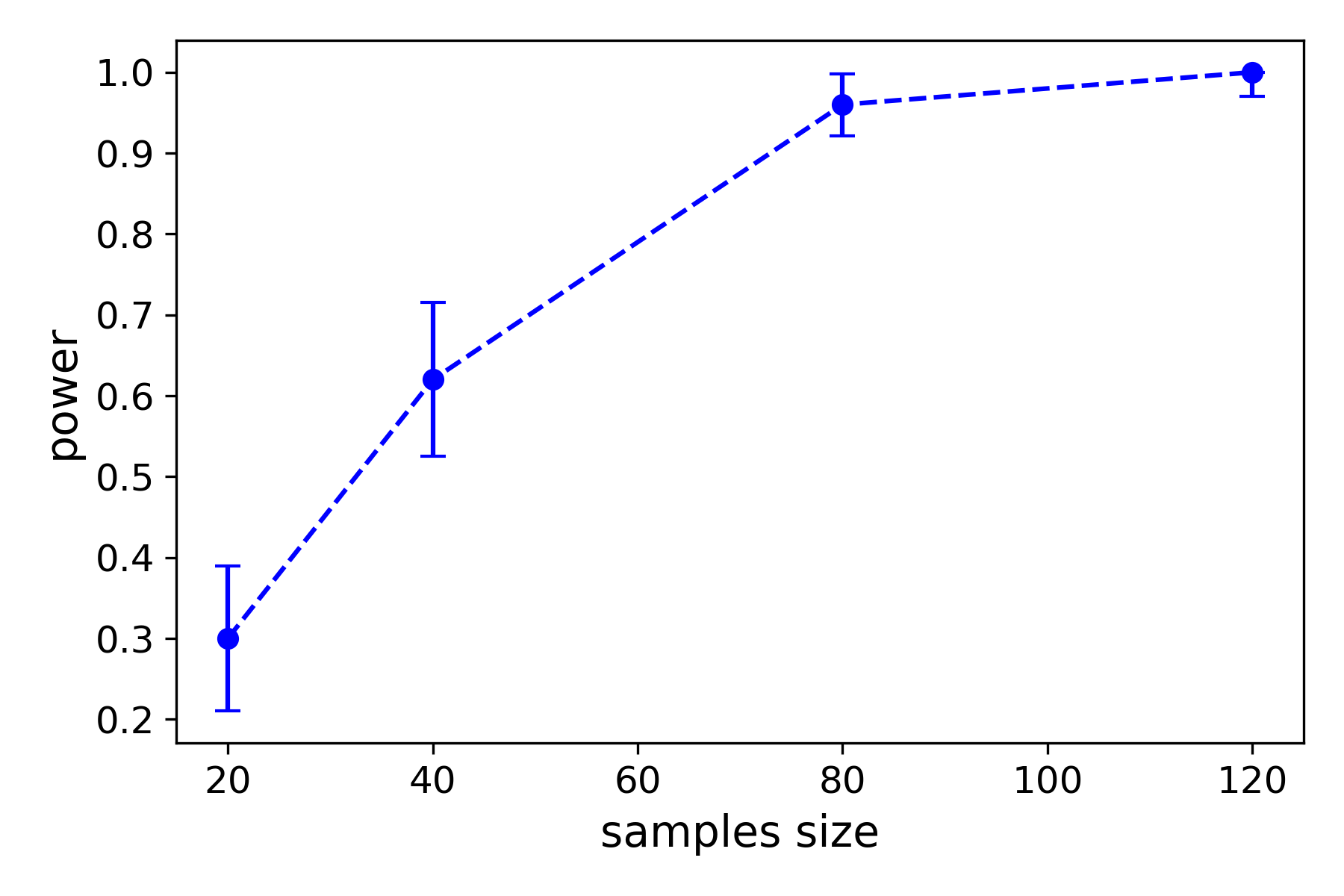}
	\caption{Powers estimations.}
	\label{subfig:test_HPD_power}
\end{subfigure}
\caption{Performances of the two-sample test using HPD processes. The level estimations are made with Disk-Disk distributions, while the power estimations are made with Disk-Disk and Disk-Annulus distributions. Graphs are unweighted and with fixed size. The samples size varies in [20,40,80,120]. Estimations are done by repeating 100 independent tests. Vertical lines represent $95\%$-confidence intervals of the estimations. }
\label{fig:test_HPD}
\end{figure}

\begin{figure}
\centering
\includegraphics[width=0.5\linewidth]{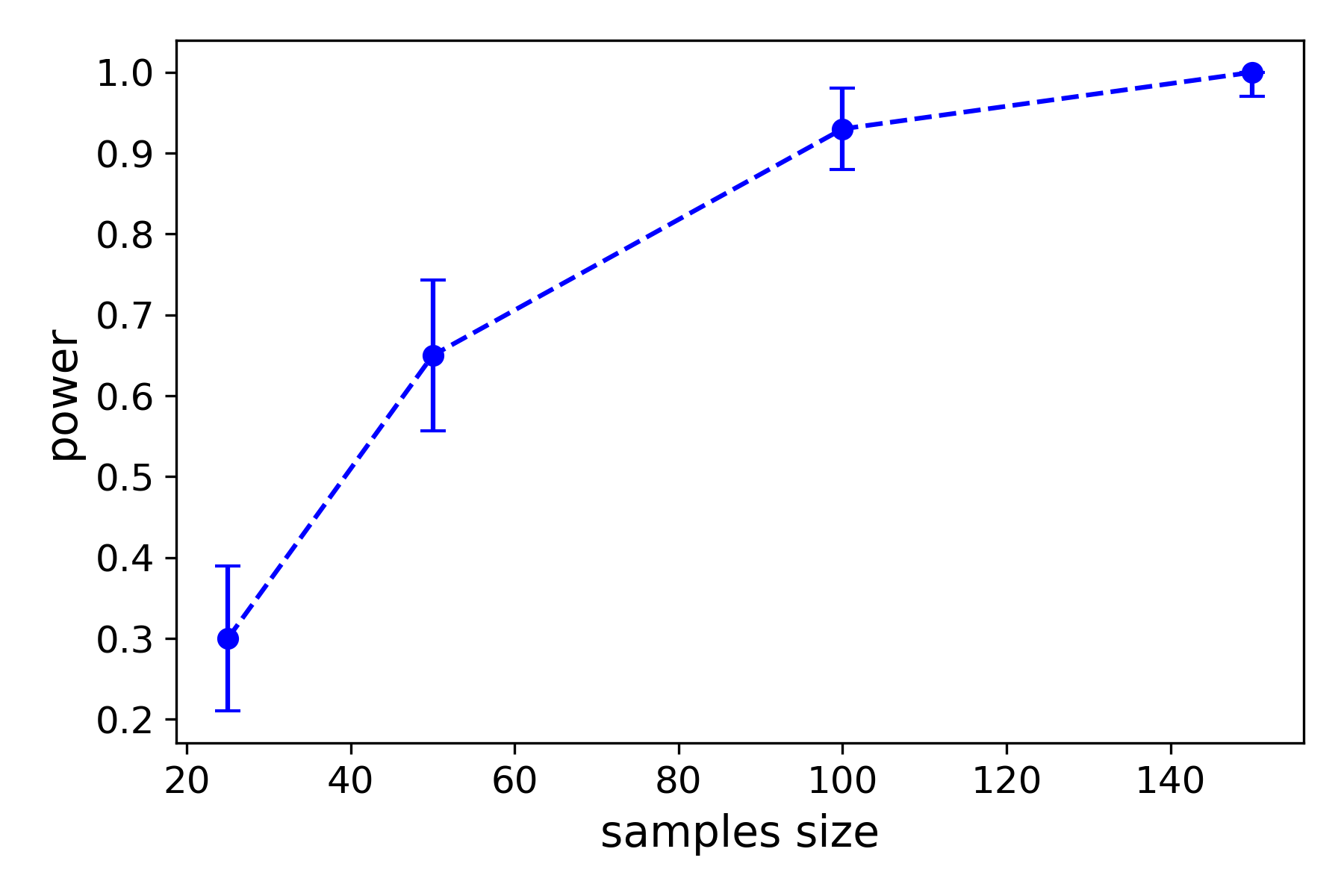}
\caption{Power of the two-sample test using HPD processes with graphs of random sizes. Power estimations are made with Disk-Disk and Disk-Annulus distributions. Graphs are unweighted and with Poissonian sizes. The samples size varies in [25,50,100,150]. Estimations are done by repeating 100 independent tests. Vertical lines represent $95\%$-confidence intervals of the estimations.}
\label{fig:test_HPD_rdm_size}
\end{figure}

\paragraph*{Comparison with other distances.} We first propose to compare our HKD-based two-sample test with tests based on other distances. We consider the \textit{Graph Diffusion Distance} (GDD) of \cite{hammond2013graph} and the comparison of graphs based on the Frobenius norm of the difference of their adjacency matrices. As these distances return real values and not processes anymore, we compare samples of pairs of graphs through the mean distance of each sample and compute a rejection threshold using a standard bootstrap method. The result of these comparisons are presented in Figure~\ref{fig:test_HKD_vs_others}.
In Figure~\ref{subfig:HKD_ERSBMw}, we consider distributions of pairs of ER-ER vs pairs of ER-SBM. In that case, both the adjacency-based and GDD-based test are unable to detect that graph distributions are different, while the HKD-based test reaches a power of 1 for samples of size greater than 100. 
In Figure~\ref{subfig:HKD_ERWSw}, we evaluate the tests on pairs of ER-ER vs pairs of ER-WS. In this setting, only the adjacency-based test fails to distinguish between the two distribution, while the HKD-based and GDD-based test have similar powers.  

It is particularly interesting to witness the power difference (and similarity) between the tests with GDD and HKD. Essentially, the GDD quickly discards the multi-scale information by taking the maximum of each HKD process. On the other hand, the HKD-based test conserves it until it compares the two mean processes. Figure~\ref{subfig:HKD_ERSBMw} illustrates that the multi-scale information contained in the HKD-processes can actually be relevant, and that it is worth conserving it to compare samples of graphs. 
However, from Figure~\ref{subfig:HKD_ERWSw}, we see that this additional multi-scale information is not always needed. A possible explanation could be that when considering SBM graphs, the density of edges is heterogeneous. For example, the density of edges in the whole graph is different compared to the one inside each block. Thus, heat diffusion may reflect that heterogeneity. To give the intuition, at the begining of the diffusion process, heat essentially diffuses inside a dense block, then it diffuses to the rest of the graph more slowly as the edge density is smaller. Thus, capturing these two phases of the heat diffusion can be meaningful. However, in the ER and WS models, graphs are essentially homogeneous. In this case, multi-scale information might be less valuable.\\

We then proceed by comparing our HPD-based two-sample test with tests based on graph kernels and the \emph{Maximum Mean Discrepancy} (MMD) \cite{gretton2012kernel}. We apply them on graphs with different sizes.  We consider the following graph kernels: 
\begin{tabular}{cccll}
$\bullet$ & GS & - & graphlet sampling & \cite{shervashidze2009efficient} \\
$\bullet$ & WL & - & Weisfeiler-Lehman & \cite{shervashidze2011weisfeiler} \\
$\bullet$ & SP & - & shortest path & \cite{borgwardt2005shortest} \\
$\bullet$ & VH & - & vertex histogram & \cite{sugiyama2015halting} \\
$\bullet$ & ST & - & SVM Theta & \cite{johansson2014global} \\
$\bullet$ & NH & - & neighborhood Hash & \cite{hido2009linear} \\
$\bullet$ & RW & - & random walk & \cite{sugiyama2015halting}.
\end{tabular}
Kernel computations were done using the Grakel python library \cite{siglidis2020grakel}. 

In this set of experiments, we only consider unweighted graphs as most of these kernels do not handle edge weights. For kernels requiring vertex labels, we choose to provide the degree of a vertex as its label. Moreover, MMD-based tests essentially perform a test on the mean of each sample seen in the corresponding RKHS. Therefore, they are not designed to handle sample of pairs of graphs. As a consequence let us give a description of the experimental procedure. We want to evaluate how well two-sample tests distinguish between two different graph distributions $P$ and $Q$. HPD-based test will be applied to samples of independent pairs $P$-$P$ and $P$-$Q$ while MMD-based tests are computed with samples of graphs from $P$ and from $Q$. 
Figure~\ref{subfig:HPD_DiskAnnulus} displays the powers the various tests for $P$ being the Disk distribution and $Q$ the Annulus distribution, each with Poissonian graph sizes. The HPD-based test achieve good performance for sample sizes greater than 100 pairs of graphs, but it is out-performed by the tests based on the graphlet sampling, Weisfeiler-Lehman, vertex histogram and SVM Theta kernels. Moreover, the three other kernel-based tests are unable to distinguish between $P$ and $Q$. Note that among the successful kernels, the WF and ST require labels and thus their tests are actually based on degree distributions, which are different in $P$ and $Q$.  On Figure~\ref{subfig:HPD_ERSBM}, the test powers are given for $P$ being the ER distribution and $Q$ the SBM one, each with Poissonian graph sizes. Here, only the GS and ST methods successfully distinguish between $P$ and $Q$, all other methods fail. 

\begin{figure}
\centering
\begin{subfigure}{0.45\linewidth}
	\includegraphics[width=\linewidth]{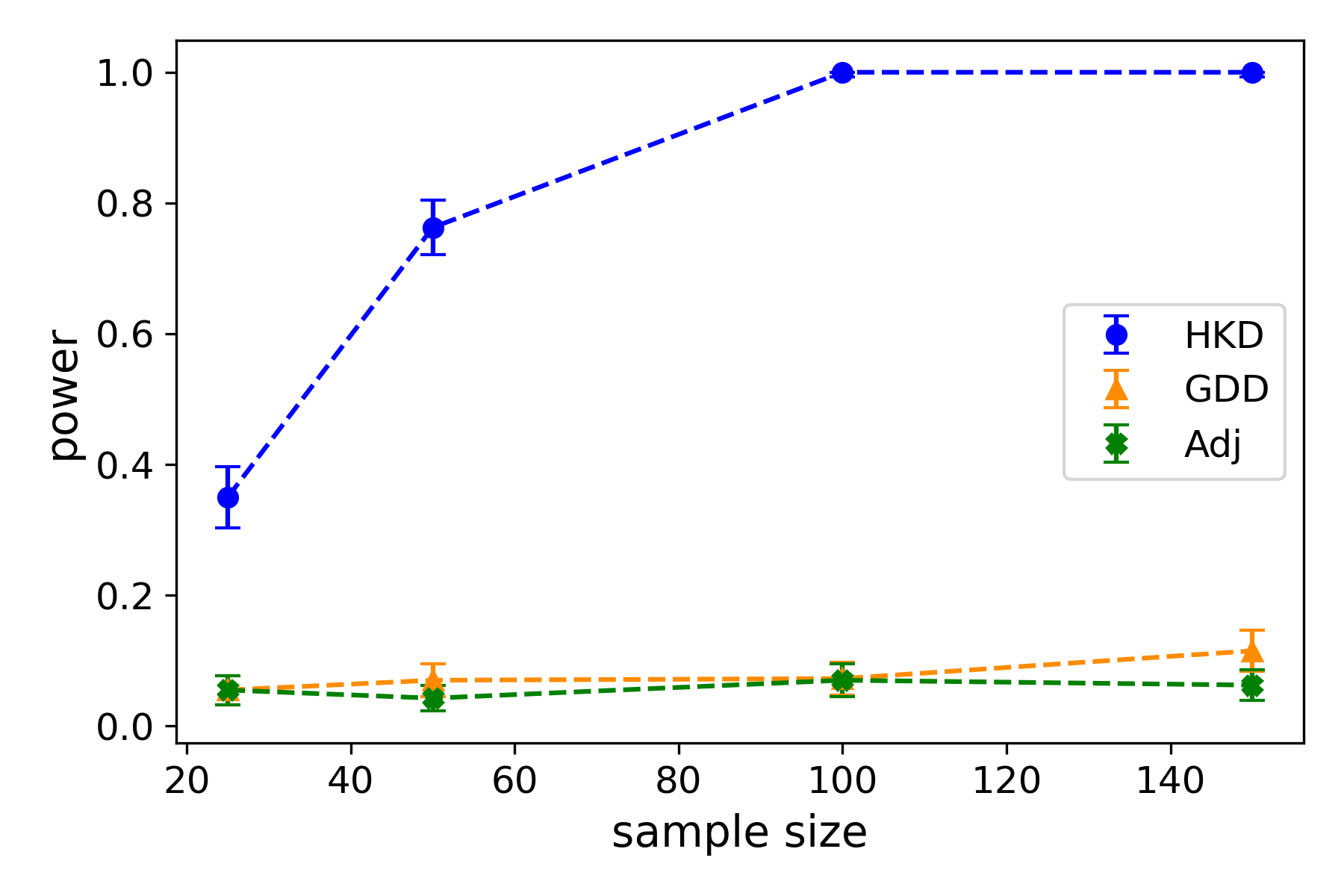}
	\caption{ER vs SBM (weighted)}
	\label{subfig:HKD_ERSBMw}
\end{subfigure}
\begin{subfigure}{0.45\linewidth}
	\includegraphics[width=\linewidth]{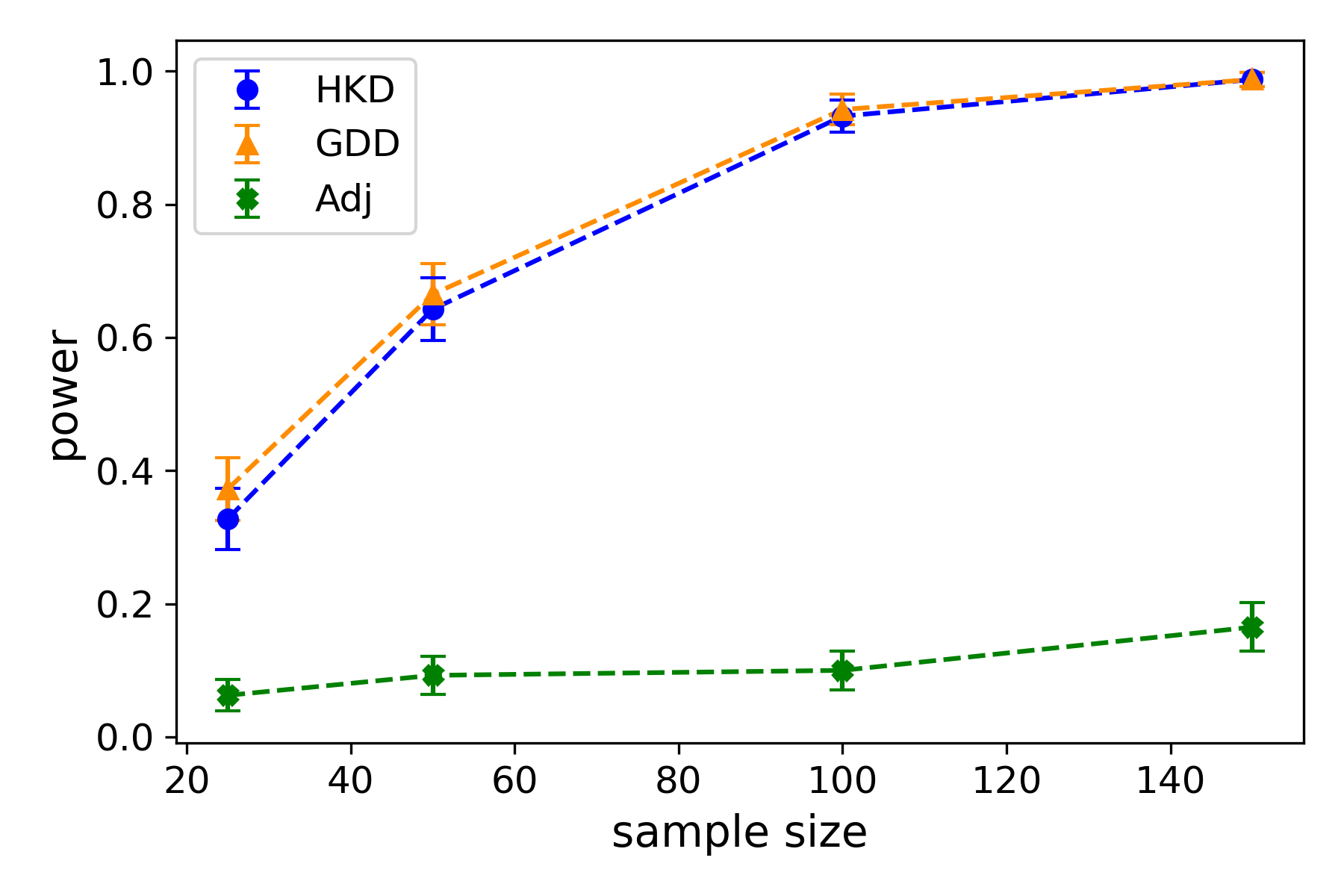}
	\caption{ER vs WS (weighted)}
	\label{subfig:HKD_ERWSw}
\end{subfigure}
\caption{Powers of the HKD-based test, the GDD-based test and the adjacency-based test. Tests are performed either under weighted ER-ER and ER-SBM distributions (left) or weighted ER-ER and ER-WS distributions (right). The samples size varies in [25,50,100,150]. Estimations are done by repeating 400 independent tests. Vertical lines represent $95\%$-confidence intervals of the estimations.}
\label{fig:test_HKD_vs_others}
\end{figure}

\begin{figure}
\centering
\begin{subfigure}{0.45\linewidth}
	\includegraphics[width=\linewidth]{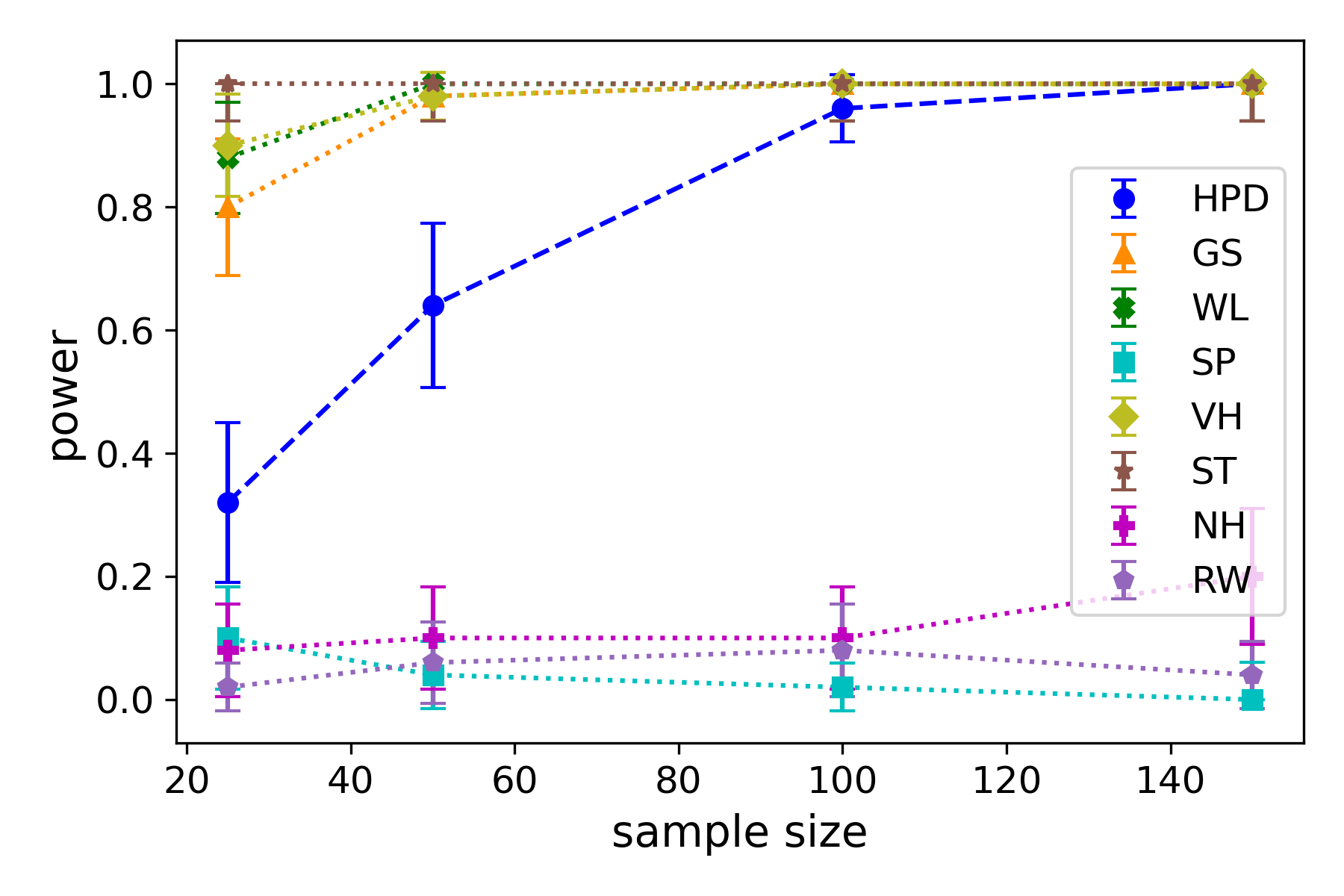}
	\caption{Disk vs Annulus (random size)}
	\label{subfig:HPD_DiskAnnulus}
\end{subfigure}
\begin{subfigure}{0.45\linewidth}
	\includegraphics[width=\linewidth]{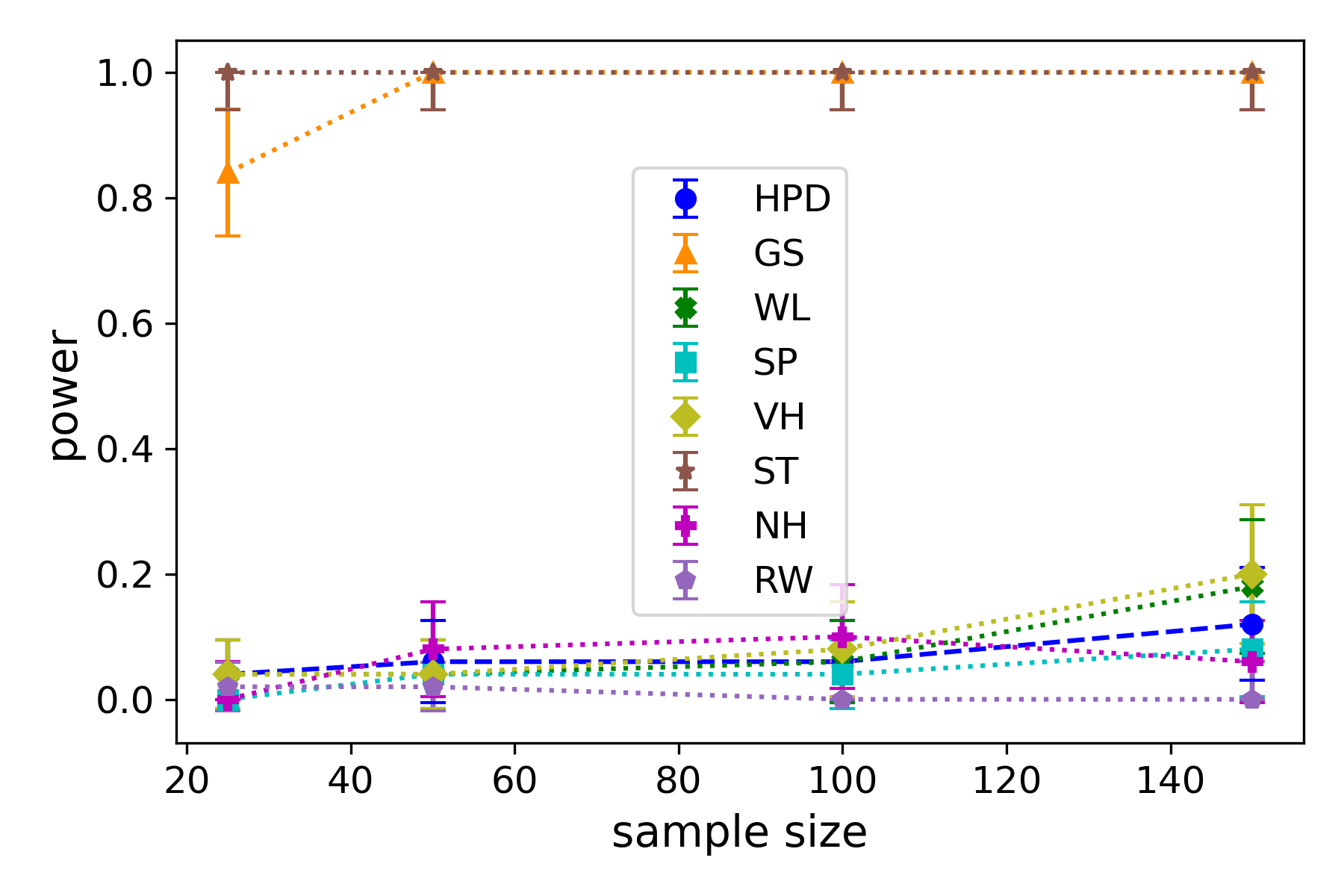}
	\caption{ER vs SBM (random size)}
	\label{subfig:HPD_ERSBM}
\end{subfigure}
\caption{Powers of the HPD-based test and MMD tests for various graph kernels. Tests are performed either under Disk-Disk and Disk-Annulus distributions with graphs of random sizes (left) or ER-ER and ER-SBM distributions with graphs of random sizes (right). The samples size varies in [25,50,100,150]. Estimations are done by repeating 50 independent tests. Vertical lines represent $95\%$-confidence intervals of the estimations.}
\label{fig:test_HPD_vs_others}
\end{figure}

\paragraph*{Comparison with the Neyman-Pearson Test.}

The Neyman-Pearson test \cite{neyman1933ix} is an optimal testing procedure, in the sense that it is the test with the highest power for a given level. But being based on likelihood ratios, it rarely is computable. Its performances are determined by the total variation distance (TV) between the two distributions. If we consider the case of distributions that depend on the sample size, the speed at which their TV tends to 0 determines if the Neyman-Pearson test will asymptotically be able to distinguish between the two distributions.
For ER models, independence of the edges allows to theoretically compute bounds of the TV and determine the phase transition. Let us take a real number $p$ such that $0<p<1$. And for each sample size $N \geq 1$ consider the parameters $p_0(N)$ and $p_1(N)$, such that both converge to $p$. Following \cite[Section 13.1.]{lehmann2006testing}, we can show that as long as $\abs{p_0(N) - p_1(N)} \gg N^{-1/2}$, the Neyman-Pearson test asymptotically distinguishes between the distributions of independent pairs of $\ER(n,p_0(N))$ and independent pairs of $\ER(n,p_1(N))$ when using $N$-samples.
Figure~\ref{fig:HKD_test_ER_neyman_pearson} shows that for large enough sample sizes, our two-sample test based on HKD processes distinguishes between ER models up to $\abs{p_0(N) - p_1(N)} = C \log N / \sqrt{N}$, with $C=0.01$. The tests are computed with a set level of 0.05 and with 1000 bootstrap samples.

\begin{figure}
\centering
\includegraphics[width=0.5\linewidth]{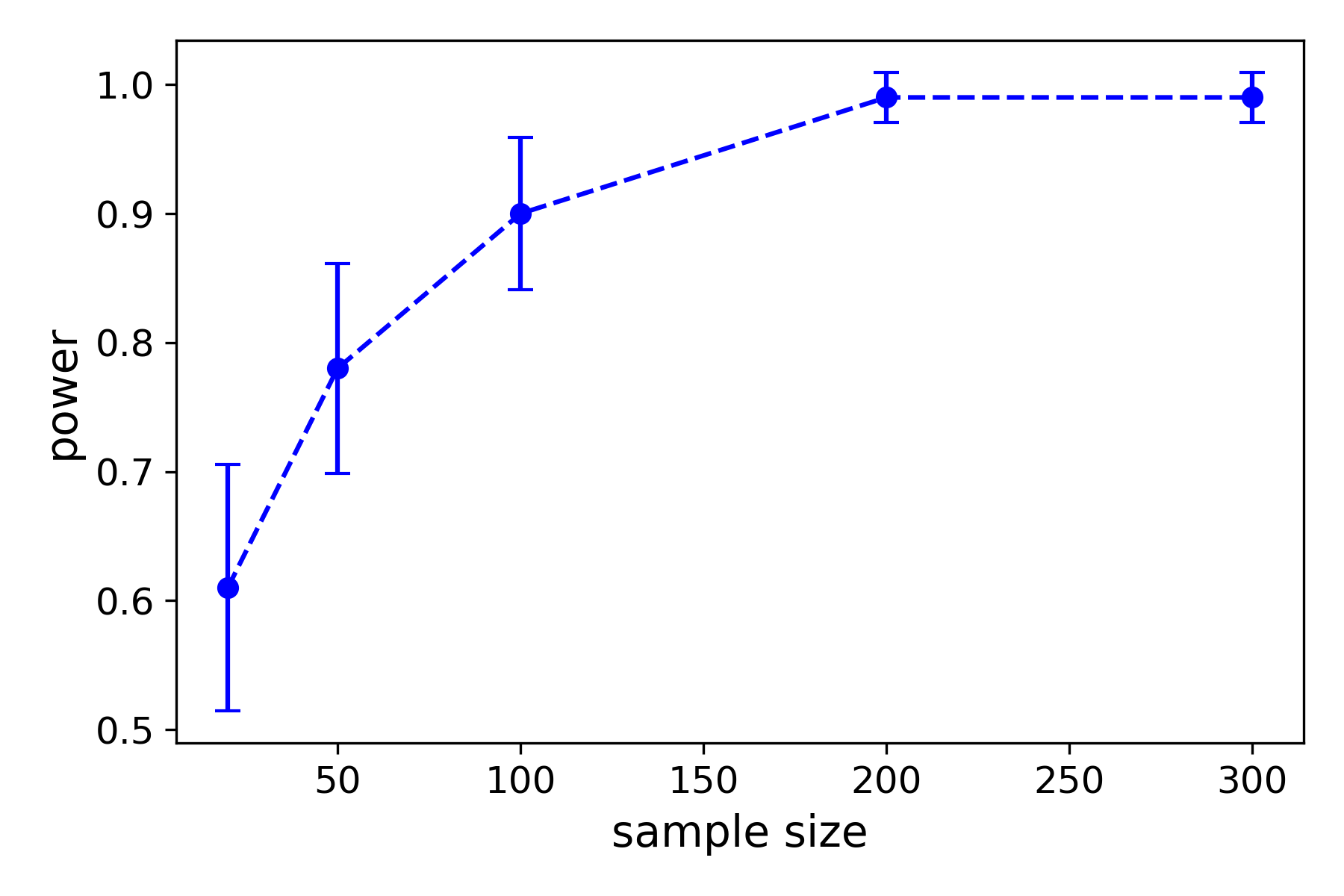}
\caption{Evolution of the power of the HKD two-sample test when $\abs{p_0(N) - p_1(N)} = C \log N / \sqrt{N}$. The samples size $N$ varies in [20,50,100,200,300]. Estimations are done by repeating 100 independent tests. Vertical lines represent $95\%$-confidence power intervals of the estimations. }
\label{fig:HKD_test_ER_neyman_pearson}
\end{figure}

\subsection{Computational complexity}
Both Algorithm~\ref{alg:conf_band} and Algorithm~\ref{alg:two_sample_test} require to compute eigenvalues and eigenvectors decompositions of Laplacian matrices in order to obtain heat diffusion distance processes. For symmetric matrices of size $n \times n$, the computational complexity of \textit{divide-and-conquer} algorithms to compute such decomposition is, in the worst case, in $O(n^3)$ operations.  (see Table~3.1 of \cite{nakatsukasa2013stable}).

In our experiments, we restricted graph sizes to around 50 vertices. This is because we needed to repeat the tests several hundred times to estimate the levels and powers with statistical significance. However, single tests with graphs of a few hundreds or even thousands of vertices could be considered. 

In the case where graphs would be too large for eigendecompositions, strategies to approximate the heat kernel could be developed. Previous works used Taylor expansion, Chebychev approximation, or low frequency approximation \cite{tsitsulin2018netlsd, marcotte2022fast}.

\section{Conclusion}

We proposed two multiscale comparisons of graphs using heat diffusion processes, namely the HKD and HPD. The first one requires the assumption of equal graph sizes and a known NC, while the second one is free of these assumptions. The multiscale approach solves the problem of choosing an informative diffusion time. We proposed to use these processes to analyze data sets of pairs of graphs and were able to design consistent confidence bands and two-sample tests. The methods are supported by theoretical results: the HKD and HPD families are Donsker, meaning that the processes verify a functional central limit theorem. Moreover, the processes admit Gaussian approximations with rates that are independent of the graph sizes. These results are very general and could easily be applied to other processes. Essentially, the processes are required to be uniformly bounded and Lipschitz-continuous. Moreover, the performances of our methods were illustrated by simulations on synthetic data sets. We showed that the two-sample tests were able to distinguish between Erd\H{o}s-Rényi and SBM graphs, as well as between geometric graphs sampled on different domains. Moreover, we confirmed the benefit of using multi-scale notion of distances, as the HKD and HPD processes, over other notions of distances, especially when graphs have heterogeneous densities. On Erd\H{o}s-Rényi models with parameters depending on the sample size, the tests were still distinguishing between the different distributions, even when working close to the phase transition of the Neyman-Pearson test.   \medskip

As future work, we would like to extend these methods to be able to perform learning tasks, e.g., clustering, classification, or change point detection on graph data. Extensions to be able to deal with data sets of graphs by opposition to pairs of graphs should also be developed to broaden the application spectrum. 
Transferring our approach to directed graphs could also be interesting. Although technical difficulties related to spectral decomposition of non-symmetric matrices would need to be dealt with.
On the theoretical side, studying the interplay between graph sizes and sample sizes could be a first step toward non-asymptotic methods to analyze data sets of graphs. \medskip

Nonetheless, we believe that the introduction of the HKD and HPD processes has the potential to bring innovative and statistically founded ways to analyze data sets of graphs. Moreover, the theoretical results presented in this work being very general, our methods could be extended to other fields.

\section*{Appendix}

\begin{appendix}

\section{Proof of Theorem~\ref{theo:donsker}}
\label{sec:proof_donsker}

We start by proving the Donsker property. 
\begin{proof}
We need to prove the weak convergence of $\{G_N f_t, \ t \in I\}$ to the centered gaussian process $\bbG$.
To do so, remark that assumption \bound ensures that second moments of $f(X)$ are finite. This allows to apply the multidimensional version of the central limit theorem to all finite-dimensional marginals of the random function $G_N f$. Hence, these finite-dimensional marginals converge in distribution to those of $\bbG$.
To conclude to the weak convergence of the whole process, the sequence of random functions $\{ G_N f, \ N \geq 1 \}$ needs to be tight. According to \cite[Theorem 12.3]{billingsley2013convergence}, tightness of $\{ G_N f, \ N \geq 1 \}$ is implied by the following two conditions:
\begin{enumerate}
\item There exists $t_0 \in I$, such that $\{ G_N f_{t_0}, \ N \geq 1 \}$ is tight.
\item There exist $\gamma \geq 0$, $\alpha > 1$ and a non-decreasing function $\psi : I \to \R$ such that $\forall t,s \in I$, $\forall N \geq 1$,
\begin{equation*}
\expect{}{\abs{ G_N f_t - G_N f_s }^\gamma} \leq \abs{\psi(t) - \psi(s)}^{\alpha}.
\end{equation*}
\end{enumerate}

Let us start by proving point 1. Let $t_0$ be any point in $I$.
We have to prove that for all $\eta > 0$, there exists $\alpha > 0$ such that for all $N \geq 1$
\begin{equation}
P\left(  \abs{G_N f_{t_0}} \leq \alpha   \right) > 1- \eta.
\label{eq:tight}
\end{equation}
If $Var(f_{t_0}(X)) = 0$, $\{ G_N f_{t_0}, \ N \geq 1 \}$ is tight, since for all $N$, $G_N f_{t_0} = 0$ $P$-\as.
Otherwise, fix $\eta > 0$. As the left hand side in (\ref{eq:tight}) is non-decreasing with respect to $\alpha$, we may as well show that there exists $\alpha$ such that for all $N$, $P\left(  -\alpha < G_N f_{t_0} \leq \alpha   \right) > 1- \eta$. Following from \bound, $f_{t_0}(X)$ admits a third moment. Combined with the positive variance, we can apply the Berry-Essen Theorem: if $F_N$ and $\phi$ denote the cumulative distribution functions of, respectively, $G_N f_{t_0}$ and a centered Gaussian variable of variance $Var(f_{t_0}(X))$, then there exists a constant $C$ such that for all $\alpha \in \R$ and all $N \geq 1$
\begin{equation*}
\abs{ F_N(\alpha) - \phi(\alpha)} \leq \frac{C}{\sqrt{N}}.
\end{equation*}
Now we take $N_\eta$ such that for all $N>N_\eta$, $C / \sqrt{N} < \eta / 4$, and we choose $\alpha_\eta$ such that $\phi(\alpha_\eta) > 1 - \eta / 4$. Then, for all $N > N_\eta$
\begin{align*}
  & P\left(  -\alpha_\eta < G_N f_{t_0} \leq \alpha_\eta   \right) \\
= & F_N(\alpha_\eta) - F_N(-\alpha_\eta) \\
= & \phi(\alpha_\eta) - \phi(-\alpha_\eta) + \left(F_N(\alpha_\eta) - \phi(\alpha_\eta)  \right) -  \left(F_N(-\alpha_\eta) - \phi(-\alpha_\eta)  \right) \\
\geq & \phi(\alpha_\eta) - \phi(-\alpha_\eta) - 2\frac{C}{\sqrt{N}} \\
= & 2 \phi(\alpha_\eta) - 1 - 2\frac{C}{\sqrt{N}} \\
> & 2 \left(1 - \frac{\eta}{4} \right) -  1 - 2 \frac{\eta}{4} = 1 - \eta
\end{align*}
One can easily choose $\alpha > \alpha_\eta$ to extend the inequality $P\left(  -\alpha_\eta < G_N f_{t_0} \leq \alpha_\eta   \right) > 1 - \eta$ to all $N \geq 1$. This finishes the proof of point 1.

The proof of point 2, is a consequence of assumption \lipsch.
For all $t,s \in I$, one has the following inequality
\begin{align*}
&\expect{}{\abs{ G_N f_t - G_N f_s }^2} \\
=& \expect{}{ \abs{(f_t(X) - Pf_t) - (f_s(X) - Pf_s)}^2 } \\
\leq & (2k\abs{t-s})^2  \\ 
=& ( 2kt - 2ks )^2.
\end{align*}
This proves point 2. with $\gamma =2$, $\alpha=2$ and $\psi(u) = 2ku$, and finishes the proof of Theorem~\ref{theo:donsker}.
\end{proof}

The proof of Gaussian approximation is based on a result from \cite{berthet2006revisiting}. They derive rates for the Gaussian approximation of more general processes. They do so by approaching the process by finite-dimensional marginals and applying a multidimensional Gaussian approximation result. Controlling the covering number of the family (or equivalently its metric entropy) allows them to derive a good trade-off between a good approximation by the marginals and keeping the dimension small enough to obtain good rates in the multidimensional Gaussian approximation. For the sake of completeness, let us recall their result.

Let $\M$ be the set of all measurable real-valued functions on $(\bbX, \X)$. The authors work with a centered and scaled empirical process $\{G_N \tilde{f} , \ \tilde{f} \in \tilde{\F} \}$ indexed by a general family $\tilde{\F} \subset \M$, not necessarily indexed by $I$. Their result shows that, under mild assumptions, this process can be approached by a centered Gaussian process $\tilde{\bbG}$ indexed by $\tilde{\F}$ with covariance $\expect{}{\tilde{\bbG}(\tilde{f}) \tilde{\bbG}(\tilde{g})} = P\tilde{f} \tilde{g} - P \tilde{f} P \tilde{g}$, for $\tilde{f}, \tilde{g} \in \tilde{\F}$.
Let us define the covering number of $\tilde{\F}$.
First, assume that there exists $\tilde{F} \in \M$ such that for all $\tilde{f} \in \tilde{\F}$ and all $x \in \bbX$, $\vert \tilde{f}(x) \vert \leq \tilde{F}(x)$. We say that $\tilde{F}$ is an \textit{envelope} of $\tilde{\F}$.  For all probability measures $Q$ on $(\bbX, \X)$, we consider the semi-metric $d_Q(g,h)^2 = \int (g-h)^2 dQ$, for $g,h \in \M$. Under $d_Q$, we define the ball of radius $\delta >0$ centered in $h \in \M$ by $B_Q(h,\delta) := \{g \in \mathcal{M}, \ d_Q(g,h) < \delta\}$. We also define $\tilde{F}_Q^2 := \int \tilde{F}^2 dQ$. For $\delta>0$, let $N(\tilde{\F}, d_Q, \delta)$ be the size of the smallest finite subset $K \subset \M$ verifying that the union of balls $B_Q(h,\delta)$ for $h$ in $K$ covers $\tilde{\F}$.
Finally, we set  the covering number of $\tilde{\F}$ to be $N(\tilde{\F}, \delta) := \sup_{Q} N(\tilde{\F}, d_Q, \delta . \tilde{F}_Q)$, where the supremum is taken over all $Q$ such that $0 < \tilde{F}_Q < \infty$.

Before stating the result of \cite{berthet2006revisiting}, consider these two basic assumptions.

\hspace{0.0\linewidth}
\begin{minipage}{0.95\linewidth}
\begin{itemize}
\item[(F.i)] For some $M>0$ and for all $\tilde{f} \in \tilde{\F}$, $\sup _{x \in \bbX} \vert \tilde{f}(x) \vert \leq M / 2$.
\item[(F.ii)] The class $\tilde{\F}$ is point-wise measurable, i.e., there exists a countable subclass $\tilde{\F}_\infty$ of $\tilde{\F}$ such that we can find for any function $\tilde{f} \in \tilde{\F}$ a sequence of functions $\{\tilde{f}_m\}$ in $\tilde{\F}_\infty$ for which $\lim _{m \rightarrow \infty} \tilde{f}_m(x)=\tilde{f}(x)$ for all $x \in \bbX$.
\end{itemize}
\end{minipage}

\begin{proposition}{\cite[Proposition 1.]{berthet2006revisiting}}
Assume that $\tilde{\F}$ verifies (F.i) and (F.ii). Take $\tilde{F} := M/2$ as the envelope of $\tilde{\F}$. Moreover, assume that there exist positive constants $c_0$ and $v_0$ such that
$N(\tilde{\F}, \delta) \leq c_0 \delta^{-v_0}$. Then, for each $\lambda>1$ there is a constant $\rho(\lambda)$ such that for each $N$, one can construct $X_1, \dots, X_N$ and $\tilde{\bbG}^{(N)}$ such that
\begin{equation*}
\proba{}{\sup\limits_{\tilde{f} \in \tilde{\F}} \abs{G_N \tilde{f} - \tilde{\bbG}^{(N)}(\tilde{f})} > \rho(\lambda) N^{-\frac{1}{2 + 5v_0}} \log N^{\frac{4+5v_0}{4+10v_0}} } \leq N^{-\lambda}.
\end{equation*}
\label{prop:berthet}
\end{proposition}

Let us now prove the Gaussian approximation in Theorem~\ref{theo:donsker} by applying the previous proposition.

\begin{proof}
The proof consists in applying Proposition~\ref{prop:berthet} to $\F$ with $v_0 = 1$. Clearly, assumption \bound implies (F.i). Since the paths $t \to f_t(x)$ are continuous for all $x \in  \bbX$, taking $\F_\infty = \{ f_t, \ t \in I \cap \Q \}$ gives (F.ii), where $\Q$ is the set of rational numbers.

Let us prove the upper-bound on the covering number.
Let $L$ be the length of $I$. From \lipsch, we know that there exists a constant $k$, such that $t \to f_t(x)$ is $k$-lipschitz, for all $x$. Considering a regular grid on $I$ , $\min I = t_0 < t_1 < \dots <t_q = \max I$. For all $t \in I$ there exists a integer $j$ such that for all $x \in \bbX$, $\abs{ f_t(x) - f_{t_j}(x) } \leq k \abs{t-t_j} \leq kL/q$. Hence, for all probability measures $Q$, $d_Q(f_t, f_{t_j}) \leq kL/q$, meaning that $N(\F, d_Q, kT/q) \leq q+1$. As this last inequality stands for all $Q$, we can find a constant $c_0>0$ such that $N(\F, \delta) \leq c_0 \delta^{-1}$, for all $\delta>0$. This finishes the proof.
\end{proof}

\section{Application to the Heat distance processes.}
\label{sec:HDdonsker_proof}

\subsection{Bounds for the laplacian eigenvalues.}

Let us start by proving a lemma on the extremal laplacian eigenvalues of graphs in $\Gn$. We define $\Lambda_{\min}$ and $\Lambda_{\max}$ as, respectively, the minimal and maximal positive laplacian eigenvalues of graphs in $\Gn$:
\begin{align*}
\Lambda_{\min} &:= \inf \{ \lambda > 0, \ \textnormal{s.t. } \lambda \textnormal{ is an eigenvalue of } L(G),  G \in  \Gn \} \\
\Lambda_{\max} &:= \sup \{ \lambda > 0, \ \textnormal{s.t. } \lambda \textnormal{ is an eigenvalue of } L(G),  G \in  \Gn \}.
\end{align*}

\begin{lemma} $\Lambda_{\min}$ and $\Lambda_{\max}$ satisfy the following bounds:
\begin{align}
\Lambda_{\min} & \geq \frac{8 w_{\min}}{ n^2 }, \label{eq:lambda_min}\\
\Lambda_{\max} & \leq n w_{\max} .
\end{align}
\label{lem:extreme_egval}
\end{lemma}
Note that (\ref{eq:lambda_min}) will not be used in the rest of the paper. Still, we choose to present it as we believe it could be of future use. 

\begin{proof}
Take $G \in \Gn$, we want to prove that $\lambda_{min}$, the smallest positive eigenvalue of $L(G)$, verifies
$ \lambda_{min} \geq 8 w_{\min} n^{-2}$. To do so, we apply a Cheeger-type inequality. First assume that $G$ is connected, hence $\lambda_{min} = \lambda_2(G)$. Let $V$ be the set of vertices of $G$. Following \cite[Section 3]{fiedler1995estimate}, we define the \textit{average minimal cut} of $G$ by
\begin{equation*}
\gamma(G) = \min_{\emptyset \neq U \subsetneq V} \sum\limits_{i \in U, \ j \in V \backslash U} \frac{  w_{i,j} }{ \abs{U} ( \nnodes - \abs{U}) }.
\end{equation*}
As $G$ is connected, there exists at least one edge with a weight greater than $w_{\min}$ joining $U$ and $V \backslash U$. Hence,
\begin{equation*}
\sum\limits_{i \in U, \ j \in V \backslash U} w_{i,j} \geq w_{\min}.
\end{equation*}
Moreover, for all $U$, $\abs{U} ( \nnodes - \abs{U}) \leq \nnodes^2 / 4$. This yields to $\gamma(G) \geq 4 w_{\min}/n^2$. From \cite[Theorem 2]{fiedler1995estimate}, we have that $\lambda_2(G) \geq 2 \gamma(G)$, hence $\lambda_2(G) \geq 8 w_{\min} n^{-2}$.
If now $G$ is not connected, one can check that there exists $G_{sub}$ a connected subgraph of $G$ of size $\nnodes_{sub}$, such that $\lambda_{min} = \lambda_2(G_{sub})$. This gives
\begin{equation*}
\lambda_{min} = \lambda_2(G_{sub}) \geq \frac{8 w_{\min}}{ n_{sub}^2} \geq \frac{8 w_{\min}}{ n^2 }.
\end{equation*}
This finishes the proof of the first bound.

We now prove the second bound. The largest eigenvalue of L(G) is denoted by $\lambda_\nnodes$. Letting $x$ be an eigenvector associated to $\lambda_\nnodes$ verifying $\|x\|_2 = 1$, we have
\begin{equation*}
\lambda_n = x^T L(G) x = \sum\limits_{i<j} w_{i,j} (x_i - x_j)^2 \leq w_{\max}  \sum\limits_{i<j} (x_i - x_j)^2 = w_{\max} \  x^T L(G_{comp}) x
\end{equation*}
where $G_{comp}$ is the complete undirected graph, hence
\begin{equation*}
L(G_{comp}) = \begin{pmatrix}
\nnodes-1 & -1      & \cdots & -1 \\
-1        & \ddots  & \ddots & \vdots \\
\vdots    & \ddots  & \ddots & -1 \\
-1        & \cdots  & -1     & \nnodes-1
\end{pmatrix}.
\end{equation*}
One can check that $x^T L(G_{comp}) x \leq \lambda_n(L(G_{comp}))=\nnodes$, so $\lambda_n \leq n w_{\max}$.
\end{proof}

\subsection{Result on HKD processes.}
\label{sec:HKDdonsker_proof}

To apply Theorem~\ref{theo:donsker} to the HKD processes we first prove that they are uniformly bounded and Lipschitz-continous.

\begin{proposition} For all $G,G'$ in $\Gn$, denote by $(\lambda_k)_{1 \leq k \leq \nnodes}$ and $(\phi_k)_{1 \leq k \leq \nnodes}$ (\textit{resp.} $(\lambda_l')_{1 \leq l \leq \nnodes}$ and $(\phi_l')_{1 \leq l \leq \nnodes}$) the eigenvalues and orthonormal eigenvectors of $L(G)$ (\textit{resp.} $L(G')$). Remember that we always choose $\phi_1$ and $\phi_1'$ equal to the vector with all entries equal to $1/\sqrt{\nnodes}$. We denote by $\langle \cdot, \cdot \rangle$ the standard scalar product in $\R^\nnodes$. Then, the application ${t \to D_t((G,G'))}$ verifies:
\begin{enumerate}
\item For all  $t \in \segm$, $D_t((G,G'))$ can be written in terms of the eigen-elements of $L(G)$ and $L(G')$:
\begin{equation}
 D_t((G,G')) = \left(  \sum\limits_{k,l = 2}^\nnodes  (e^{-t\lambda_k} - e^{-t \lambda_l'})^2 \langle \phi_k, \phi_l' \rangle^2   \right)^{1/2} .
\label{eq:hkd_expression}
\end{equation}
\item For all $t \in \segm$, $D_t((G,G')) \leq \sqrt{\nnodes}$.
\item $t \mapsto D_t((G,G'))$ is $(\nnodes^{3/2} w_{\max})$-Lipschitz continuous on $\segm$.
\end{enumerate}
\label{prop:hkd}
\end{proposition}

\begin{proof}
Let $G,G'$ be in $\Gn$.
We start by proving (\ref{eq:hkd_expression}). This is done through the following computation:
\begin{align*}
&\|e^{-tL} - e^{-tL'}\|_F^2 \\
=& \|e^{-tL}\|_F^2 + \|e^{-tL'}\|_F^2 - 2 Tr\left( e^{-tL} e^{-tL'}   \right) \\
=& \sum\limits_{k=1}^\nnodes e^{-2t \lambda_k} \| \phi_k \phi_k^T \|_F^2 + \sum\limits_{l=1}^\nnodes e^{-2t \lambda'_l} \| \phi'_l {\phi'_l}^T \|_F^2 - 2 \sum\limits_{k,l = 1}^\nnodes e^{-t \lambda_k} e^{-t \lambda'_l} Tr \left( \phi_k \phi_k^T \phi'_l {\phi'_l}^T \right)
\end{align*}
One can prove that $Tr \left( \phi_k \phi_k^T \phi'_l {\phi'_l}^T \right) = \langle \phi_k, \phi'_l \rangle^2$. \\
Similarly $\| \phi_k \phi_k^T \|_F^2 = \| \phi_k \|_2^2 = \sum\limits_{l=1}^\nnodes \langle \phi_k, \phi'_l \rangle^2$ and $\| \phi'_l {\phi'_l}^T \|_F^2 = \| \phi'_l \|_2^2 = \sum\limits_{k=1}^\nnodes \langle \phi_k, \phi'_l \rangle^2$. So
\begin{align*}
\|e^{-tL} - e^{-tL'}\|_F^2 =& \sum\limits_{k,l = 1}^\nnodes  (e^{-t\lambda_k} - e^{-t \lambda_l'})^2 \langle \phi_k, \phi_l' \rangle^2 .
\end{align*}
The sums can start at $k=2$ and $l=2$ thanks to the facts that $\lambda_1 = \lambda'_1 = 0$ and $\phi_1 = \phi'_1$ combined with the orthogonality of the eigenvectors families. This finishes the proof of (\ref{eq:hkd_expression}).

Bounding all terms $(e^{-t\lambda_k} - e^{-t \lambda_l'})^2$ by 1 in (\ref{eq:hkd_expression}), and using the orthonormality of the eigenvectors families yields to $D_t((G,G')) \leq \sqrt{\nnodes}$.

We now prove the Lipschitz result. One can check that $t \to D_t((G,G'))$ is $\C^1$ on $\segm$. The case $G=G'$ is easily dealt with. Assume now that $G \neq G'$. For all $t \in (0,T]$,
\begin{align*}
\abs{\frac{d}{dt} D_t((G,G'))} =& \abs{ \frac{ -\sum\limits_{k,l=2}^\nnodes  (e^{-t \lambda_k}-e^{-t\lambda_l'})(\lambda_k e^{-t \lambda_k}-\lambda_l' e^{-t\lambda_l'})\langle \phi_k, \phi_l' \rangle^2}{\left( \sum\limits_{k,l=2}^\nnodes  (e^{-t \lambda_k}-e^{-t\lambda_l'})^2\langle \phi_k, \phi_l' \rangle^2 \right)^{1/2}}} \\
\leq& \sum\limits_{k,l=2}^\nnodes \frac{ \abs{e^{-t \lambda_k}-e^{-t\lambda_l'}} \abs{\lambda_k e^{-t \lambda_k}-\lambda_l' e^{-t\lambda_l'}}\abs{\langle \phi_k, \phi_l' \rangle}^2}{\left( \sum\limits_{i,j=2}^\nnodes  (e^{-t \lambda_i}-e^{-t\lambda_j'})^2\langle \phi_i, \phi_j' \rangle^2 \right)^{1/2}} \\
 \leq & \sqrt{\sum\limits_{k,l=2}^\nnodes  \abs{\lambda_k e^{-t \lambda_k}-\lambda_l' e^{-t\lambda_l'}}^2 \langle \phi_k, \phi_l' \rangle^2}
\sqrt{ \frac{ \sum\limits_{k,l=2}^\nnodes (e^{-t \lambda_k}-e^{-t\lambda_l'})^2 \langle \phi_k, \phi_l' \rangle^2}{\sum\limits_{i,j=2}^\nnodes  (e^{-t \lambda_i}-e^{-t\lambda_j'})^2\langle \phi_i, \phi_j' \rangle^2 }} \\
=& \sqrt{\sum\limits_{k,l=2}^\nnodes  \abs{\lambda_k e^{-t \lambda_k}-\lambda_l' e^{-t\lambda_l'}}^2 \langle \phi_k, \phi_l' \rangle^2},
\end{align*}
where the last inequality comes from the Cauchy-Schwarz inequality.

According to the mean value theorem, for all $k$ and $l$
\[ \abs{\lambda_k e^{-t \lambda_k}-\lambda_l' e^{-t\lambda_l'}} \leq \abs{\lambda_k - \lambda_l'} \leq \Lambda_{\max}. \]

Hence, for all $t\in (0,T]$,
\[
\abs{\frac{d}{dt} D_t((G,G'))} \leq \Lambda_{\max} \sqrt{\sum\limits_{k,l=2}^\nnodes  \langle \phi_k, \phi_l' \rangle^2} \leq n w_{\max} \sqrt{n} = n^{3/2}w_{\max}.\]
This finishes the proof.
\end{proof}

The proof of Theorem~\ref{theo:HKDdonsker} is a direct application of Theorem~\ref{theo:donsker}, as Proposition~\ref{prop:hkd} ensures that the hypotheses of Theorem~\ref{theo:donsker} are verified by the HKD process.

\subsection{Result on HPD processes.}
\label{sec:HPDdonsker_proof}

To prove Theorem~\ref{theo:HPDdonsker}, we start by showing that under the hypotheses of the theorem the HPD processes are uniformly bounded and Lipschitz-continuous.

\begin{proposition}
For all $G, G'$ in $\Gun$, the application $t \to H_t((G,G'))$ verifies:
\begin{enumerate}
\item For all $t \in \segm$ , $0 \leq H_t((G,G')) \leq 1$.
\item $t \mapsto H_t((G,G'))$ is $(2 \nnodes w_{\max})$-Lipschitz-continuous on $\segm$.
\end{enumerate}
\label{prop:hpd}
\end{proposition}

\begin{proof}
Recall that for all $G \in \Gun$, and for a vertex $i$,
\[h_t(G)(i) =  \sum\limits_{k = 1}^{\nnodes(G)} e^{-t \lambda_k} \phi_k(i)^2, \]
where $(\lambda_k)_{1 \leq k \leq \nnodes(G)}$ and $(\phi_k)_{1 \leq k \leq \nnodes(G)}$ are the eigenvalues and orthonormal eigenvectors of $L(G)$ and $\nnodes(G)$ is the size of $G$.
Hence $0 \leq h_t(G)(i) \leq 1$ for all $i$, meaning that all points in the diagram $\dg(G,h_t(G))$ are contained in $[0,1]^2$. So from the definition of the Bottleneck distance, $0 \leq H_t((G,G')) \leq 1$, for all $G,G'$ in $\Gun$ and for all $t \in \segm$.

Let us now compute the first derivative of $h_t(G)(i)$:
\[ \frac{d}{dt} h_t(G)(i) = - \sum\limits_{k = 1}^{\nnodes(G)} \lambda_k e^{-t \lambda_k} \phi_k(i)^2. \]
Its absolute value is upper-bounded by $\lambda_{\nnodes(G)}$, the largest eigenvalue of $L(G)$.
From Lemma~\ref{lem:extreme_egval}, we have $\lambda_{\nnodes(G)} \leq \nnodes(G) w_{\max}$. Hence, $t \to h_t(G)$ is $(\nnodes w_{\max})$-Lipschitz continuous on $\segm$. To conclude, we come back to the definition of the HPD in terms of distance between persistence diagrams. Applying the triangular inequality to the Bottleneck distance gives for all $G, G' \in \Gun$, for all $t, t' \in \segm$, and for all diagram construction $\dg$,
\begin{align*}
& \abs{\db(\dg(G,h_t(G)) , \dg(G',h_t(G'))) -  \db(\dg(G,h_{t'}(G)) , \dg(G',h_{t'}(G'))) } \\
\leq & \db(\dg(G,h_{t}(G)) , \dg(G,h_{t'}(G))) + \db(\dg(G',h_{t}(G')) , \dg(G',h_{t'}(G'))).
\end{align*}
Similarly, the same inequality but with maxima over the diagram constructions holds:
\begin{align*}
& \abs{H_t((G,G')) -  H_{t'}((G,G')) } \\
\leq & \max \db(\dg(G,h_{t}(G)) , \dg(G,h_{t'}(G))) +  \max \db(\dg(G',h_{t}(G')) , \dg(G',h_{t'}(G'))).
\end{align*}
Applying Theorem~\ref{theo:bottleneck_stability} and using the Lispchitz continuity of the HKS yields
\begin{align*}
& \abs{H_t((G,G')) -  H_{t'}((G,G')) } \\
\leq &  \|h_t(G) - h_{t'}(G) \|_\infty + \| h_{t}(G') - h_{t'}(G')\|_\infty \\
\leq & 2 n w_{\max} \abs{t-t'}.
\end{align*}
\end{proof}

Theorem~\ref{theo:HPDdonsker} follows from Proposition~\ref{prop:hpd} and Theorem~\ref{theo:donsker}.

\end{appendix}

\section*{Acknowledgments}
I am thankful to Frédéric Chazal\footnote{\label{inria}Inria Saclay} and Pascal Massart\footnote{\label{ups}Université Paris-Saclay} for the valuable discussions and advice on this work. 
I also want to thank the people from Datashape\cref{inria} and LMO\cref{ups} for the enriching conversations.
\bibliographystyle{acm} 
\bibliography{bibliography}       


\end{document}